\documentclass{article}

\usepackage[utf8]{inputenc}
\usepackage{amsmath,amssymb,amsthm, xcolor, enumitem, comment, todonotes}
\usepackage{url, hyperref}
\usepackage[all]{xy}
\usepackage[english]{babel}

\newtheorem{definition}{Definition}[section]

\newtheorem{question}[definition]{Question}

\newtheorem{remark}[definition]{Remark}

\newtheorem{theorem}{Theorem}[section]
\newtheorem{claim}[definition]{Claim}

\newtheorem{lemma}[definition]{Lemma}
\newtheorem{corollary}[definition]{Corollary}

\newtheorem{Corollary}[definition]{Corollary}

\newcommand{\po}{\mathbb{P}}
\newcommand{\qo}{\mathbb{Q}}
\newcommand{\la}{\langle}
\newcommand{\ra}{\rangle}

\newcommand{\name}{\dot}

\newcommand{\uhr}{\upharpoonright}

\newcommand{\power}{\mathcal{P}}

\newcommand{\E}{\bar{E}}

\newcommand{\ro}{\mathbb{R}}

\DeclareMathOperator{\dom}{dom}

\DeclareMathOperator{\cf}{cof}

\DeclareMathOperator{\supp}{supp}

\title{The number of normal measures, revisited}
\author{Eyal Kaplan\\UC Berkeley\\Berkeley, CA}

\begin{document}
	\maketitle

\begin{abstract}
	We present a new version of the Friedman–Magidor theorem: for every measurable cardinal $\kappa$ and $\tau\leq\kappa^{++}$, there exists a forcing extension $V\subseteq V[G]$ such that any normal measure $U\in V$ on $\kappa$ has exactly $\tau$ distinct lifts in $V[G]$, and every normal measure on $\kappa$ in $V[G]$ arises as such a lift. This version differs from the original Friedman–Magidor theorem in several notable ways. First, the new technique does not involve forcing over canonical inner models or rely on any fine-structural tools or assumptions, allowing it to be applied in the realm of large cardinals beyond the current reach of the inner model program. Second, in the case where $\tau\leq \kappa^+$, all lifts of a normal measure $U\in V$ on $\kappa$ to $V[G]$ have the same ultrapower. Finally, the technique generalizes to a version of the Friedman–Magidor theorem for extenders. An additional advantage is that the forcing used is notably simple, relying only on nonstationary support product forcing.
\end{abstract}
	
The analysis of measures in forcing extensions plays a central role in the theory of large cardinals. A major question in this area was whether the existence of a model of ZFC with exactly two normal measures follows from the consistency of ZFC with a measurable cardinal. This question was answered affirmatively by the following landmark theorem of Friedman and Magidor.

	\begin{theorem}(Friedman–Magidor theorem, \cite{FriMag09}) \label{Theorem: Friedman–Magidor} 
			Suppose $V=L[U]$ is the minimal inner model in which there is a measurable cardinal $\kappa$, and let $\tau\leq \kappa^{++}$. Then in a cardinal-preserving forcing extension, $\kappa$ is measurable and carries exactly $\tau$ normal measures.
	\end{theorem}

In their work, Friedman and Magidor identified the various components that contribute to the construction of normal measures in forcing extensions. Their forcing notion includes, for each such component, a specialized feature designed to control and constrain it. This naturally led to the orchestration of several forcing techniques, each playing an essential role in the proof. Indeed, Theorem \ref{Theorem: Friedman–Magidor} stands out for the range of methods it brings together, including forcing over canonical inner models, nonstationary support iterations, generalized Sacks forcing, and interleaved self-coding posets. These techniques have since enabled a variety of further applications, including \cite{BenNeria2016MitchellOrder1}, \cite{BenNeria2016MitchellOrder2}, \cite{BenNeriaGitik2017uniquenormalmeasureandviolationofGCHoptimally}, \cite{ApterCummingsHamkinsLargeCardinalsFewMeasures}, and \cite{BenNeriaKaplan2024KunenLike}.

The central result of this paper is a significantly simpler and more transparent proof of the Friedman–Magidor theorem. Our approach dispenses entirely with inner model theory, self-coding, and generalized Sacks forcing. Instead, it relies solely on a product forcing with nonstationary support.
	
Similarly to the original Friedman–Magidor proof, we concentrate on the case where the expected number $\tau$ of normal measures is at most $\kappa^{+}$; achieving $\tau = \kappa^{++}$ is done by performing the Kunen-Paris forcing (see \cite{kunenParis1971boolean}). 
	
	The main theorem of this paper is the following:

	\begin{theorem}\label{Theorem: controlling normal measures}
		Suppose that $\kappa$ is a measurable cardinal and $\tau \leq \kappa^{+}$. There exists a poset $\ro^\tau$ which preserves $\kappa, \kappa^{+}$ and the measurability of $\kappa$, such that, whenever $G\subseteq \ro^\tau$ is generic over $V$, the following properties hold:
		\begin{enumerate}
			\item Every normal measure $U\in V$ on $\kappa$ has exactly $\tau$ distinct lifts $\la U^*_\eta \colon \eta<\tau \ra$ to normal measures on $\kappa$ in $V[G]$, all of which have the same ultrapower over $V[G]$.
			\item Every normal measure $W\in V[G]$ on $\kappa$ has the form $U^*_\eta$ for some normal $U\in V$ and $\eta<\tau$.
		\end{enumerate} 
		Also, assuming GCH in $V$, $\ro^{\tau}$ preserves cardinals.
	\end{theorem}
	
	In the case where $\tau<\kappa$, $\ro^\tau$ has the form $\text{Cohen}(\omega)* \po^\tau$, where $\po^\tau$ consists of partial functions $f\colon \kappa\to \tau$ whose support is nowhere stationary, ordered by inclusion.\footnote{The purpose of the initial Cohen forcing is  to ensure that $\ro^{\tau}$ has a gap below $\kappa$ (see Theorem \ref{Theorem: Hamkins Gap forcing thm}).} In the case where $\tau\in \{ \kappa, \kappa^+ \}$, $\po^\tau$ is slightly modified using canonical functions, see definition \ref{Definition: the splitting forcing} below.

	As mentioned above, one of the main advantages of Theorem \ref{Theorem: controlling normal measures} is that it does not involve forcing over canonical inner models. In the original Friedman–Magidor proof, forcing over $L[U]$ played a significant role in two ways. First, by assuming that $L[U]$ is the core model and forcing over it, every normal measure on $\kappa$ in the generic extension restricts to an iterated ultrapower of $L[U]$. (This is the starting point of a sequence of classical results in inner model theory, culminating in a well-known theorem of Schindler, see \cite{schindler2006iterates}). Second, canonical inner models possess fine structure, which Friedman and Magidor utilized by incorporating self-coding posets into their forcing construction. (See \cite{FriMag09} for further details on self-coding and its role in constraining the number of normal measures).

	While these techniques demonstrate the utility of inner model theory and fine structure in forcing theory, forcing over canonical inner models can be restrictive—particularly because it does not apply at levels of the large cardinal hierarchy not yet reached by the inner model program. Since our proof of Theorem \ref{Theorem: controlling normal measures} avoids the use of inner model theory, it can be employed to obtain consistency results at the level of strongly compact and supercompact cardinals, such as the following:

	\begin{theorem}[Apter, K., Poveda]\label{Theorem: Consistency strengths at the level of supercompacts with Arthur and Alejandro}
		Assume the consistency of a supercompact cardinal and the linearity of the Mitchell order. 
		\begin{enumerate}
			\item It is consistent that there exists a supercompact, and there are exactly two normal measures on the least measurable cardinal. 
			\item Assume in addition a measurable cardinal above the supercompact. Then it is consistent that there is a supercompact cardinal, and the least measurable cardinal above it carries exactly two normal measures.
		\end{enumerate}
		\end{theorem}
		
		Theorem \ref{Theorem: Consistency strengths at the level of supercompacts with Arthur and Alejandro} and its proof are expected to appear in a joint work with Apter and Poveda. The first clause follows directly from Theorem \ref{Theorem: controlling normal measures} and the fact that the forcing $\ro^\tau$ (for $\tau = 2$) is small relative to the least supercompact cardinal above $\kappa$. The second is obtained by forcing with a variation of $\po^\tau$ after the universe has been suitably prepared so that the supercompactness of $\kappa$ is indestructible under $\kappa$-directed-closed forcing.

	Our forcing notions $\po^\tau$ (for $\tau< \kappa$) have several additional applications. In Section \ref{Section: Lifting iterates of U}, we classify all possible lifts of $U^n$ to the generic extension obtained by forcing with $\po^\tau$, where $U \in V$ is a normal measure on $\kappa$ and $n < \omega$. When this analysis is carried out in the case where $V = L[U]$ is the core model, it allows for a complete description of the $\sigma$-complete ultrafilters in the extension. This yields a ZFC model with an interesting measure structure:

	\begin{theorem}\label{Theorem: Weak UA}
		Assume that $V = L[U]$, and let $\kappa$ be the measurable cardinal. Fix $\tau<\kappa$. Let $G\subseteq \po^\tau$ be generic over $V$. Then in $L[U][G]$:
		\begin{enumerate}
			\item there are $\tau$ normal measures on $\kappa$, $\la U^*_\eta \colon \eta<\tau \ra$, all of them have the same ultrapower.
			\item for every $\sigma$-complete ultrafilter $W\in L[U][G]$, the ultrapower $\text{Ult}(L[U][G],W)$ is a finite iterated ultrapower of $L[U][G]$ by normal measures.
			\item every ultrapower of $L[U][G]$ (via a $\sigma$-complete ultrafilter) is isomorphic to $\text{Ult}(L[U][G], \left(U^*_0\right)^n)$ for some $n<\omega$.
		\end{enumerate}
	\end{theorem}
	
	Theorem \ref{Theorem: Weak UA} is related to the Weak Ultrapower Axiom. Recall that Goldberg's Ultrapower Axiom (UA) states that for any pair of $\sigma$-complete ultrafilters $U, W$, there are $\sigma$-complete ultrafilters $U^*\in M_W$ and $W^* \in M_U$,\footnote{Here and below, for a $\sigma$-complete ultrafilter $U$, we denote by $M_U$ the transitive collapse of $\text{Ult}(V,U)$, and by $j_U \colon V\to M_U$ the corresponding ultrapower embedding.} such that $M^{M_W}_{U^*} = M^{M_U}_{W^*}$ and $j_{W^*}\circ j_U = j_{U^*} \circ j_W$. The Weak Ultrapower Axiom is obtained from UA by removing the demand that $j_{W^*}\circ j_U = j_{U^*} \circ j_W$. Theorem \ref{Theorem: Weak UA} (for $1<\tau<\kappa$) demonstrates that a model of the Weak UA and the negation of UA\footnote{The negation of UA holds in the model of Theorem \ref{Theorem: Weak UA} since, by a theorem of Goldberg, UA implies the linearity of the Mitchell order. See \cite{GoldbergUABook} for more details.} can be constructed assuming the consistency of a measurable cardinal. This slightly improves a previous result by Goldberg and Benhamou (see \cite{BenhamouGoldberg2024applicationsofMagidor}), who constructed such a model assuming the consistency of a measurable limit of measurables.

	\begin{corollary}
		Weak UA + $\neg$UA is equiconsistent with a single measurable cardinal.
	\end{corollary}
	
	Finally, building on the techniques introduced earlier, we generalize the Friedman–Magidor framework from normal measures to extenders in Theorem \ref{Theorem: Friedman–Magidor for extenders}. A typical application of this theorem is illustrated by the following example:
	
	\begin{theorem}\label{Theorem: the number of extenders on a strong}
		Assume GCH and let $\kappa$ be a $(\kappa+2)$-strong-cardinal. Then for every $\tau\leq \kappa^{++}$, there exists a cardinal-preserving forcing extension $V[G]$ in which every $(\kappa, \kappa^{++})$-extender $E$ that witnesses $\kappa$ being $(\kappa+2)$-strong, lifts to exactly $\tau$ nonequivalent\footnote{Two extenders are equivalent if they induce the same ultrapower embedding} $(\kappa, \kappa^{++})$-extenders of $V[G]$, and all of them have the same ultrapower and witness $(\kappa+2)$-strongness over $V[G]$. Moreover, every $(\kappa, \kappa^{++})$-extender witnessing that $\kappa$ is $(\kappa+2)$-strong in $V[G]$ arises as such a lift.
	\end{theorem} 
	
	Goldberg observed that the bound $\kappa^{++}$ on the number of $(\kappa, \kappa^{++})$-extenders having the same ultrapower is optimal under GCH (see Theorem \ref{Theorem: Number of extenders with the same ultrapower}). Similar results can be obtained for extenders with higher degrees of strongness.

	The structure of this paper is as follows: In Section \ref{Section: NS supports} we cover the basic properties of nonstationary support products. In Section \ref{Section: The splitting forcing} we prove Theorem \ref{Theorem: controlling normal measures}. In Section \ref{Section: Lifting iterates of U} we discuss lifting iterates of $U^n$ for a normal measure $U$ and $n<\omega$, and prove Theorem \ref{Theorem: Weak UA}. In Section \ref{Section: Friedman–Magidor for extenders} we phrase and prove a Friedman–Magidor theorem for extenders (Theorem \ref{Theorem: Friedman–Magidor for extenders}), and apply it to prove Theorem \ref{Theorem: the number of extenders on a strong}. In Section \ref{Section: Open Qeustions} we discuss several open questions.

	Throughout the paper, we use the following convention in forcing: if a condition $p$ extends a condition $q$, we write $q \leq p$, and the weakest condition in a forcing notion $\po$ is denoted by $0_{\po}$. The rest of the notation is standard.
	
	\subsection*{Acknowledgments}
	The author owes special debt to Omer Ben-Neria for suggesting crucial ideas in the proof of Theorems \ref{Theorem: controlling normal measures} and \ref{Theorem: Weak UA}. The author also thanks Gabe Goldberg for numerous discussions on the subject of this paper. Goldberg identified a mistake in an earlier attempt to prove Theorem \ref{Theorem: controlling normal measures}, which led the author to discover a coding-free proof of the Friedman–Magidor theorem. However, this proof was more complicated, required an additional forcing, and resulted in a version that lacked the property that all normal measures in the generic extension have the same ultrapower. Ben-Neria suggested the key idea that led to the uniqueness of normal ultrapowers, and also to the current simplified version of the forcing in Definition \ref{Definition: the splitting forcing} below. The author is deeply grateful to both of them for their support and encouragement throughout the preparation of this work.
	
	The author would like to thank Moti Gitik for pointing out the important correction that $\po^\tau$ should be replaced by $\ro^\tau$ in Theorem \ref{Theorem: controlling normal measures}.

	\section{Nonstationary support products}\label{Section: NS supports}

	In this section, we briefly develop the basic framework of product forcing with a nonstationary support. All the ideas already appeared in one form or another in \cite{FriMag09}. Our presentation is influenced from  \cite{benneriaunger2017homogeneouschanges} and \cite{ApterCummingsNormalMeasuresOnTallCards}.
	
	\begin{definition}
	A set of ordinals $A$ is \textit{nowhere stationary}\footnote{We will adopt this terminology, although the phrase \textit{nonstationary in inaccessibles} would be more accurate in this context. For our purposes, this distinction is irrelevant, since the nowhere stationary sets we consider will consist solely of inaccessible cardinals - for such sets, being nonstationary in inaccessibles implies being nonstationary in every ordinal of uncountable cofinality.} if for every inaccessible cardinal $\alpha$, $A\cap \alpha$ is nonstationary in $\alpha$.    
	\end{definition}
	
	\begin{definition}\label{Def: nonstationary support product}(Products with nonstationary support)
		Fix a cardinal $\kappa$ and let $I\subseteq \kappa$ be an unbounded set of ordinals. For every $\alpha\in I$, let $\qo_\alpha$ be a forcing notion. The \textbf{nonstationary support product} $\po = \prod^{NS}_{\alpha\in I} \qo_\alpha$ consists of conditions $p$ which are functions with domain $I$, such that for every $\alpha\in I$, $p(\alpha)\in \qo_\alpha$, and the set
		$$ \supp(p) = \{  \alpha\in I \colon p(\alpha)\neq 0_{\qo_\alpha}   \} $$
		is nowhere stationary.
		
		The ordering is defined by setting that $p$ extends $q$ if for every $\alpha\in \supp(q)$, $p(\alpha)$ extends $q(\alpha)$ in $\qo_\alpha$ (in particular,  $\supp(q)\subseteq \supp(p)$).
	\end{definition}
	
	Assuming that $\beta<\kappa$ and $\po = \prod_{\alpha\in I}^{NS}\qo_\alpha$ is a nonstationary support product, $\po$ can be naturally factored to the form $\po\uhr \beta \times \po\setminus \beta$, where $\po\uhr \beta = \prod_{\alpha\in \beta\cap I}^{NS} \qo_\alpha $ and $\po\setminus \beta = \prod_{\alpha\in I\setminus \beta}^{NS}\qo_\alpha$, by identifying a condition $p\in \po$ with the pair $(p\uhr (I\cap \beta),  p\uhr (I\setminus \beta)) $. We denote in this case $p\uhr \beta = p\uhr (I\cap \beta)$ and $p\setminus \beta = p\uhr (I\setminus \beta)$. 
	
	The most important feature of nonstationary support products or  iterations is their fusion property.

	\begin{lemma}[The Fusion Lemma]\label{Lemma: Fusion for a product}
		Let $\kappa$ be a limit of inaccessible cardinals, and let $I\subseteq \kappa$ be an unbounded set of inaccessibles below $\kappa$. Let $\po = \prod_{\alpha\in I}^{NS}\qo_\alpha$ be a nonstationary support product. Assume that:
		\begin{enumerate}
			\item for every $\alpha\in I$, $\text{rank}({\qo}_\alpha)< \min(I\setminus \alpha+1)$.
			\item for every $\alpha\in I$, $\qo_\alpha$ is $\alpha$-closed. 
		\end{enumerate}
		Then $\po$ satisfies the \textbf{Fusion Property}; that is, given:
		\begin{itemize}
			\item a condition $p\in \po$,
			\item a sequence $\la d(\alpha) \colon \alpha<\kappa \ra$ of dense-open subsets of $\po$,
		\end{itemize}
		there exists $p^*\geq p$ and a club $C\subseteq \kappa$ (if $\cf(\kappa)= \omega$, $C$ is an unbounded cofinal $\omega$-sequence) such that, for every $\alpha\in C$, the set
		$$ \{ r\in \po\uhr {\alpha+1} \colon r^{\frown} p^*\setminus (\alpha+1)\in d(\alpha)  \} $$
		is dense in $\po\uhr {\alpha+1}$ above $p^*\uhr \alpha+1$.
	\end{lemma}
	
	\begin{proof}
		We first take care of the case where $\kappa$ is singular. Fix an inaccessible $\nu\in (\cf(\kappa),\kappa)$ and a continuous, cofinal, increasing sequence $\la \nu_i \colon i<\cf(\kappa) \ra$ in $\kappa$, such that $\nu_0 > \nu$. Since $\po\setminus (\nu+1)$ is more than $\max(\cf(\kappa), |\po\uhr(\nu+1)|)$-closed, we may find a condition $p^*\geq p$ such that for every $i<\cf(\kappa)$ and $r\in \po\uhr \nu+1$, there exists $r'\geq r$ such that ${r'}^{\frown} (p^*\setminus (\nu+1))\in d(\nu_i)$. It follows that for every $i<\cf(\kappa)$, 
		$$ \{ r\in \po\uhr \nu_i+1 \colon r^{\frown} p^*\setminus (\nu_i+1)\in d(\nu_i) \} $$
		is dense in $\po\uhr (\nu_i+1)$ above $p^*\uhr (\nu_i+1)$. Thus, $p^*$ and $C = \{ \nu_i \colon i<\cf(\kappa) \}$ are as desired.
		
		Thus, assume from now on that $\kappa$ is regular. We construct by induction sequences $\la p_i \colon i<\kappa \ra\subseteq \po$,  $\la \nu_i \colon i<\kappa \ra\subseteq \kappa$ such that $\la p_i \colon i<\kappa \ra$ is a fusion sequence with respect to $\la \nu_i \colon i<\kappa \ra$, in the sense that the following properties are fulfilled: 
		\begin{enumerate}
			\item $\la \nu_i \colon i<\kappa \ra$ is a continuous, cofinal sequence in $\kappa$.
			\item $\la p_i \colon i<\kappa \ra$ is an increasing sequence of conditions extending $p$, such that, for every $i\leq j<\kappa$,
			\begin{enumerate}
				\item $p_j \uhr (\nu_i+1) = p_i \uhr (\nu_i+1)$.
				\item $\nu_j \notin \supp(p_i)$.
			\end{enumerate}
			\item For every $i<\kappa$,
			$$ \{  r\in \po\uhr \nu_{i}+1 \colon r^{\frown} (p_{i+1}\setminus (\nu_i+1))\in d(\nu_i) \} $$
			is a dense-open subset of $\po\uhr {\nu_i+1}$.
		\end{enumerate}
		Assuming that such sequences have been be constructed, we can set $p^* = \bigcup_{i<\kappa} p_i \uhr (\nu_{i}+1)$. Since $C = \{ \nu_i \colon i<\kappa \}$ is a club disjoint from $\supp(p^*)$, $p^*$ is a legitimate condition in $\po$.\footnote{Note that we require that $\supp(p^*)$ is nonstationary in any inaccessible $\lambda<\kappa$ as well; this follows immediately since given such $\lambda$, $p^*\uhr \lambda = p_i \uhr \lambda$ for any high enough $i<\kappa$.} It's then not hard to verify that $p^*$ and $C$ are as required.
		
		Thus, it suffices to prove that sequences $\langle p_i : i < \kappa \rangle$ and $\langle \nu_i : i < \kappa \rangle$ with the desired properties can be constructed. We define these sequences inductively, along with an auxiliary sequence $\langle C_i : i < \kappa \rangle$ of clubs in $\kappa$, such that, in addition to the above requirements,
		\begin{enumerate}  \setcounter{enumi}{3}
			\item for every $i<\kappa$, $C_i$ is disjoint from $\supp(p_i)$.
			\item for every $i < j<\kappa$, $\nu_j \in C_i$ and $C_j \subseteq C_i$.
		\end{enumerate}
		We proceed to the construction. For $i=0$, take $p_0 = p$, $\nu_0 = \nu$ and $C_0$ any club disjoint from $\supp(p_0)$. 
		
		For the successor step, assume that $p_i, \nu_i, C_i$ have been constructed. First set $\nu_{i+1} = \min( C_i\setminus (\nu_i+1 ))$. Define $p_{i+1}\geq p_i$ such that:
		\begin{itemize}
			\item $p_{i+1}\uhr (\nu_{i+1}+1) = p_i\uhr (\nu_{i+1}+1)$ (in particular, $\nu_{i+1}\notin \supp(p_{i+1})$).
			\item $p_{i+1}\setminus (\nu_{i+1}+1)$ extends $p_i\setminus (\nu_{i+1}+1)$, and satisfies that
			$$  \{  r\in \po\uhr \nu_{i+1}+1 \colon r^{\frown} ( p_{i+1}\setminus ( \nu_{i+1}+1 ) )\in d(\nu_{i+1}) \} $$
			is a dense subset of $\po\uhr \nu_{i+1}+1$. 
		\end{itemize}
		Note that the amount of closure of $\po\setminus \nu_{i+1}+1$ is larger than $| \po\uhr \nu_{i+1}+1 |$, allowing the construction of such a condition $p_{i+1}$. Finally, let $C_{i+1}\subseteq C_i$ be a club in $\kappa$ disjoint from $\supp(p_{i+1})$.	
		
		For the limit step, assume that $i<\kappa$ and $\la p_j, \nu_j , C_j \colon j<i \ra$ were constructed. First, set $\nu_i = \sup\{ \nu_j \colon j<i \}$. Let $p_i\in \po$ be a condition such that:
		\begin{itemize}
			\item $p_i \uhr\nu_i = \bigcup_{j<i} p_j \uhr (\nu_j+1) $.
			\item $\nu_i \notin \supp(p_i)$.
			\item $p_i \setminus (\nu_i+1)$ extends all the conditions $\langle p_j \setminus (\nu_i+1) \colon j<i \rangle$, and satisfies that
			$$ \{ r\in \po\uhr \nu_i+1 \colon r^{\frown} ( p_i\setminus (\nu_i+1))\in d(\nu_i) \} $$ is a dense subset of $\po\setminus \nu_i+1$.
		\end{itemize}
		Note that $\po\setminus \nu_i+1$ is closed enough to fulfill the third requirement. Also, by our construction, $\nu_i$ is a limit point of each club $C_j$ for $j<i$, and thus $\nu_i \notin \supp(p_j)$ for every $j<i$. This ensures that $p_i$ extends $p_j$ for $j<i$. Finally, let $C_i\subseteq \bigcap_{j<i} C_j$ be a club disjoint from $\supp(p_i)$. This concludes the inductive construction, which concludes the proof.
	\end{proof}

\begin{corollary}\label{Corollary: evaluating new functions by old ones}
	Let $\po = \prod^{NS}_{\alpha\in I}\qo_\alpha$ be as in Lemma \ref{Lemma: Fusion for a product}, $p\in \po$ and let $\name{f}$  be a $\po$-name forced by $p$ to be a function from $\kappa$ to the ordinals. Then there exists $q\geq p$, a club $C\subseteq \kappa$ and a function $F\in V$ such that $C\subseteq \dom(F)$ and for every $\alpha\in C$, $|F(\alpha)|\leq |\po_{\alpha+1}|$ (in particular, $|F(\alpha)|< \min(I\setminus \alpha+1)$) and $q\Vdash \name{f}(\alpha)\in F(\alpha)$.
\end{corollary}	
	\begin{proof}
		Apply Lemma \ref{Lemma: Fusion for a product} on the sequence $\la d(\alpha) \colon \alpha<\kappa \ra$, where each $d(\alpha)$ is the dense-open set of conditions deciding the value of $\name{f}(\alpha)$. Let $q\geq p$ be the obtained condition and $C\subseteq \kappa$ the obtained club, and define, for every $\alpha\in C$,
		$$ F(\alpha) = \{ \gamma\in \text{ON} \colon \exists r\geq q\uhr \alpha+1 \  (r^{\frown} (q\setminus \alpha+1) )\Vdash \name{f}(\alpha) = \check{\gamma} \}. $$
	\end{proof}

	The assumptions of Lemma \ref{Lemma: Fusion for a product} suffice to ensure that $\kappa$ is preserved after forcing with $\po$. But we obtain even more.
	
	\begin{Corollary} \label{Corollary: preservation of kappa plus}
		 Let $\po = \prod^{NS}_{\alpha\in I}\qo_\alpha$ be as in \ref{Lemma: Fusion for a product}.  Then there is no cofinal function $f\colon \kappa\to \kappa^{+}$ in $V[G]$.
	\end{Corollary}
	
	\begin{proof}
		 Let $p\in \po$ and $\name{f}$ be a $\po$-name for an increasing function from $\kappa$ to $\kappa^+$. We wish to prove that $\text{Im}(f)$ is forced by an extension of $p$ to be bounded in $\kappa^+$.  By \ref{Corollary: evaluating new functions by old ones} there exists $q\geq p$ and a club $C\subseteq \kappa$ such that for every $\alpha\in C$, $q\Vdash \name{f}(\alpha)\in F(\alpha)$ and $|F(\alpha)|< \min(I\setminus (\alpha+1)) < \kappa$. Let $\gamma = \sup\{ \sup(F(\alpha))  \colon \alpha\in C \}+1$. Then $\gamma<\kappa^+$ and, since $C$ is unbounded in $\kappa$ and $\name{f}$ is increasing,  $q\Vdash \sup(\mbox{Im}(\name{f})) < \gamma$. 
	\end{proof}

	In particular, Corollary \ref{Corollary: preservation of kappa plus} implies that $\po$ preserves $\kappa^{+}$. 
	Assuming further that GCH holds in $V$, and that each participating forcing $\name{\qo}_\alpha$ preserves cardinals, we can strengthen the conclusion to full preservation of cardinals by $\po$. Without the GCH assumption, however, this need not hold (see Claim \ref{Claim: Splitting might collapse cardinals}).
	
	\begin{Corollary}\label{Corollary: if the iterands preserve cards and GCH then the entire iteration peserves cards}
	Assume GCH. Let $\po = \prod_{\alpha\in I}^{NS}\qo_\alpha$ be as in \ref{Lemma: Fusion for a product}. Assume that for every $\alpha\in I$, $\Vdash_{\po_\alpha} "\qo_\alpha \text{ preserves cardinals"}$. Then $\po$ preserves cardinals.
	\end{Corollary}

	\begin{proof}
	Assume otherwise, and let $\mu$ be the least cardinal which is not preserved. It follows that $\mu$ is a successor cardinal. We can't have $\mu\geq \kappa^{++}$ since $|\po| = \kappa^{+}$ by our assumption that $2^\kappa = \kappa^+$. So $\mu = \lambda^+$ for some $\lambda\leq \kappa$, and by Corollary \ref{Corollary: preservation of kappa plus}, we can assume that $\lambda<\kappa$. The forcing $\po\setminus (\lambda+1)$ is more than $\lambda^{+}$-closed, and, if $\lambda\in I$, the forcing $\qo_\lambda$ preserves $\lambda^+$. So if $\mu = \lambda^+$ is collapsed, this is due to $\po\uhr \lambda$. If $\lambda$ is not a limit point of $I$, then by $\mbox{GCH}$, $|\po \uhr \lambda| \leq \lambda^+$ and thus $\mu$ is preserved. So we can assume that $\lambda$ is a limit point of $I$, and $\po_\lambda$ collapses $\mu = \lambda^+$. This contradicts Corollary \ref{Corollary: preservation of kappa plus}. 
	\end{proof}
	
	We conclude this section with the following classical application of Corollary \ref{Corollary: evaluating new functions by old ones}.
	
	\begin{corollary}\label{Corollary: density of old clubs}
		Assume that $\kappa$ is regular. For every club $C\subseteq \kappa$ in $V[G]$, there exists a club $D\subseteq \kappa$ in $V$ such that $D\subseteq C$.
	\end{corollary}

	\begin{proof}
		Let $f\in V[G]$ be the increasing enumeration of $C$. Apply Corollary \ref{Corollary: evaluating new functions by old ones} to find, in $V$, a function $g\colon \kappa\to \kappa$ such that $g$ dominates $f$ on a (ground model) club. Since $f$ is increasing and $\kappa$ is regular, we can assume that $g$ is increasing an dominates $f$ everywhere. Now define $D$ to be the club of closure points of $g$. Then $D\subseteq C$, since every $\alpha\in D$ is a limit point of $C\cap \alpha$: indeed, if $i< \alpha$ then $f(i)\leq g(i)< \alpha$ and $i\leq f(i)\in C$.
	\end{proof}

\section{The splitting forcing}\label{Section: The splitting forcing}
	
	The goal of this section is to prove Theorem \ref{Theorem: controlling normal measures}.
	
	Let $\kappa$ be a measurable in $V$. Denote by $I$ the set of inaccessible cardinals below $\kappa$. Fix throughout this section an ordinal $\tau\leq \kappa^+$. For every $\eta<\kappa^+$, let $f_\eta$ be the $\eta$-th canonical function. Let $f_{\kappa^+}$ be the successor cardinal function. 
	
	Let us define the main forcing notion in this paper. We refer to it as "the splitting forcing". 
	
	\begin{definition}[The Splitting forcing]\label{Definition: the splitting forcing}
		Define	
		$$\po^\tau = \{  f\in \prod_{\alpha\in X}f_\tau(\alpha) \colon X= \dom(f) \text{ is a nowhere stationary subset of }I \}$$
		ordered by inclusion,  that is, $f$ extends $g$ if and only if $g\subseteq f$. \\
		Note that in the case $\tau < \kappa$, $\po^\tau$ admits the simpler presentation
		$$\po^\tau = \{  f\colon X \to \tau \colon X= \dom(f) \text{ is a nowhere stationary subset of }I \}.$$
	\end{definition}

	The forcing notion $\po^{\tau}$ can be viewed as a nonstationary support product  $\prod^{NS}_{\alpha\in I} \qo_\alpha$, where, for every $\alpha<\kappa$, $\qo_\alpha$ is trivial unless $\alpha$ is inaccessible, and then $\qo_\alpha =\{ 0_{\qo_\alpha} \}\cup f_\tau(\alpha)$ is the atomic forcing notion that satisfies:
	\begin{itemize}
		\item $0_{\qo_\alpha}$ is the weakest condition of $\qo_\alpha$.
		\item the ordinals in $f_\tau(\alpha)$ form a maximal antichain in $\qo_\alpha$.
	\end{itemize}
	In particular, a generic for $\qo_\alpha$ is basically a choice of an ordinal in $f_\tau(\alpha)$.

	\begin{remark}\label{Remark: limitation on tau}
		There are two main reasons for the limitation $\tau\leq \kappa^+$ in this section.
		
		\begin{enumerate}
			\item As a nonstationary-support product, $\po^{\tau}$ can be analyzed with the Fusion Lemma \ref{Lemma: Fusion for a product}. To use it, however, we must ensure that $\qo_\alpha \in V_\kappa$ for all $\alpha$. If $2^\kappa = \kappa^{+}$, then $\tau = \kappa^{++}$ cannot be represented by a function $f_\tau$ whose values lie below $\kappa$, and the forcings $\qo_\alpha$ turn out to be quite large. On the other hand, if GCH fails at $\kappa$ and $U$ is a normal measure that computes $\kappa^{++}$ correctly in its ultrapower, then forcing with $\po^{\tau}$ may collapse $\kappa^{++}$ (see Claim \ref{Claim: Splitting might collapse cardinals}).
			\item The first reason doesn't rule out the possibility that $\tau$ is any ordinal below the least inaccessible above $\kappa$ in $\text{Ult}(V,U)$, where $U$ is the normal measure we would like to lift.\footnote{Working with such $\tau$, we can ensures that $|\qo_\alpha|$ is below the amount of closure of the forcing $\po^{\tau}\setminus \alpha+1$, which is one of the assumptions of the Fusion Lemma.} Relaxing the restriction $\tau<\kappa^{+}$ makes the classification of normal measures more complicated then the one exhibited in Theorem \ref{Theorem: controlling normal measures}. We study this more general framework in section \ref{Section: Friedman–Magidor for extenders}, and produce the  classification of normal measures in Theorem \ref{Theorem: Classification of normal measures in the more general context}.
		\end{enumerate}
	\end{remark}

	\begin{lemma}
		Assume that $G\subseteq \po^\tau$ is generic over $V$. For every $\eta<\tau$, define--
		$$S_\eta = \{ \alpha\in I \colon{(\cup G)}(\alpha)= f_\eta(\alpha)  \}$$
		Then $\vec{S} = \la S_\eta \colon \eta<\tau \ra$ is a sequence of stationary subsets of $I$. If $\tau<\kappa$, $\vec{S}$ is a partition of $I$ and $V[G] = V[\vec{S}]$.
	\end{lemma}
	
	\begin{proof}
		Fix $\eta<\tau$, and let $C\in V[G]$ be a club in $\kappa$. By Corollary \ref{Corollary: density of old clubs} there exists in $V$ a club $D\subseteq \kappa$ such that $D\subseteq C$. Given a condition $p\in \po^\tau$, we can pick $\alpha\in D\setminus \supp(p)$ due to the nonstationary support, and extend $p$ to $p\cup \{ (\alpha, f_\eta(\alpha)) \}$. Thus, by density, there exists $q\in G$ and $\alpha\in D$ such that $q(\alpha) =f_\eta(\alpha)$. Thus, in $V[G]$,  $\alpha\in C\cap S_\eta$.
		
		If $\tau< \kappa$ then each $f_{\eta}(\alpha)$ simply equals $\eta$, and $S_\eta, S_{\eta'}$ for $\eta\neq \eta'$ are disjoint. Also, $G$ could be read from $\vec{S}$, since for every condition $p$, $p\in G$ if and only if
		$$\forall \alpha\in \dom(p) \  (p(\alpha) = \eta \iff \alpha\in S_\eta). $$
	\end{proof}
	
	Next, we show that every normal measure $U\in V$ on $\kappa$ generates, together with one of the stationary sets $S_\eta$, a normal measure on $\kappa$ in $V[G]$.

	\begin{lemma}\label{Lemma: lifting normal measures}
			Let $U\in V$ be a normal measure on $\kappa$. Let $G\subseteq \po^{\tau}$ be generic over $V$. Then for every $\eta<\tau$, $U\cup \{S_\eta\}$ generates a normal measure $U^*_\eta$ on $\kappa$ in $V[G]$. 	Furthermore, the $U^*_\eta$-s have the same ultrapower over $V[G]$.
	\end{lemma}

	\begin{proof}
		Since $\tau$ is fixed, we denote $\po = \po^\tau$ throughout this proof. Fix $\eta<\tau$. The forcing $j_U(\po)$ is, by elementarity, a nonstationary support product $\prod^{NS}_{\alpha\in j_U(I)}\qo^{M_U}_\alpha$, where, for each $\alpha\in j_U(I)\cap \kappa = I$, $\qo^{M_U}_\alpha = \qo_\alpha$, and for each $\alpha\in j_U(I)\setminus \kappa$, $\qo^{M_U}_\alpha$ is an atomic forcing that chooses an ordinal below $j_U(f_\tau)(\alpha)$. In particular, $\kappa\in j_U(I)$ and $\qo^{M_U}_\kappa$ is an atomic forcing that chooses an ordinal below $\tau$. Let 
		$$ H_\eta = \{ q\in j_U(\po) \colon \exists p\in G \  \left(q\subseteq j_U(p)\cup \{ \la \kappa, \eta \ra \} \right) \}. $$
		Note that for every $p\in \po$, $\kappa\notin \supp(j_U(p))$. Therefore $ j_U(p)\cup \{ \la \kappa, \eta \ra \} $ is a function and its support is a nowhere stationary subset of $j_U(\kappa)$, so $H_\eta \subseteq j_U(\po)$. Furthermore, the elements of $H_\eta$ are mutually compatible and clearly $j_U[G]\subseteq H_\eta$. We argue that $H_\eta\subseteq j_U(\po)$ is generic over $M_U$.

		Fix $E\subseteq j_U(\po)$ dense open in $M_U$. Let $\langle e(\alpha) \colon \alpha<\kappa \rangle$ be a sequence of dense open subsets of $\po$ such that $E = [\alpha\mapsto e(\alpha)]_U$. Recall that $\po$ can be viewed as a nonstationary support product. Apply the Fusion Lemma \ref{Lemma: Fusion for a product} to find a condition $q\in G$ and a club $C\subseteq \kappa$ such that for every $\alpha\in C$,
		$$\{ r\in \po_{\alpha+1} \colon r^{\frown} (q\setminus \alpha+1) \in e(\alpha) \}$$
		is dense in $\po_{\alpha+1}$. This means that, in $M_U$, the set--
		$$\{ r\in j_U(\po)_{\kappa+1} \colon r^{\frown} (j_U(q)\setminus \alpha+1)\in E \}$$
		is dense. Since $G$ is generic for $j_U(\po)_\kappa = \po_\kappa$ (over $M_U$), and $\eta$ is generic for $\qo^{M_U}_{\kappa}$ (over $M_U[G]$), there exists $p\in G$ such that $p^{\frown} \langle (\kappa, \eta) \rangle^{\frown} (j_U(q)\setminus \kappa+1)\in E $. By increasing $p$ if necessary inside $G$, we can assume that $p$ extends $q$. Since $E$ is open, we deduce that $j_U(p)\cup \{ \la \kappa, \eta \ra  \}\in E\cap H_i$, as desired.\\
		
		Since $j_U[G]\subseteq H_\eta$, we can now use Silver's criterion to lift the embedding $j_U$ to an elementary embedding from $V[G]$ to $M_U[H_\eta]$. By standard arguments (see \cite{cummings2009iteratedforcingsandelemembs}), the lifted embedding is an ultrapower embedding with respect to a normal measure $U^*_\eta$. Indeed, $U^*_{\eta}$ is the normal measure derived from the lifted embedding using $\kappa$ as a seed. In other words, for $X\in V[G]$, we have $X\in U^*_\eta$ if and only if there exists a condition in $H_\eta$ that forces that $\kappa\in j_U(\name{X})$ (where $\name{X}$ is a $\po$-name for $X$). Equivalently, there exists a condition $p\in G$ such that 
		$$ j_U(p)\cup \{ \la \kappa, \eta \ra \}\Vdash \kappa\in j_U(\name{X}). $$
		This implies that $\{ \alpha<\kappa \colon p\cup \{ \la \alpha,f_\eta(\alpha) \ra \}    \}\in U$, and
		$$ \{ \alpha<\kappa \colon p\cup \{ \la \alpha,f_\eta(\alpha) \ra \} \Vdash \check{\alpha}\in \name{X}     \}\cap \{ \alpha<\kappa \colon (\cup G)(\alpha) = f_\eta(\alpha) \}\subseteq X $$
		so $U\cup \{ S_i \} $ generates $U^*_i$.
		
		Note that for every $\eta,\zeta<\tau$, $H_\eta, H_\zeta$ are mutually definable: one is obtained from the other by changing the generic value at $\kappa$.\footnote{The fact that we used a product rather than an iteration matters here: if we forced with an iteration, and $q'\in H_\eta$ is obtained from $q$ by replacing $\la \kappa, \eta \ra$ (assuming that $q\uhr \kappa$ forces that $\la \kappa, \eta \ra$ belongs to $q$) with $\la \kappa, \zeta \ra$, the fact that $q\subseteq j_U(p)\cup \{ \la \kappa, \eta \ra \}$ for some $p\in G$ does not necessarily imply that $q'\subseteq j_U(p')\cup \{\la \kappa, \zeta \ra\}$ for some $p'\in G$. Using products ensures that $q\setminus \kappa+1$ doesn't depend on $q(\kappa)$, and thus $q'$ is compatible with $H_\zeta$.} This implies that all the measures $U_\eta$ ($\eta<\tau$) have the same ultrapower. 
		\end{proof}

		We now proceed to prove that every normal measure of $V[G]$ has the form $U^*_\eta$ for some $\eta<\tau$ and normal measure $U\in V$. The main tool that assists bypassing inner-model-theoretic arguments here, is Hamkins' Gap forcing theorem.
		
		\begin{theorem}[Hamkins' Gap Forcing theorem \cite{hamkins2001gap}]\label{Theorem: Hamkins Gap forcing thm}
		Assume that $\po$ is a forcing notion which has a gap at some cardinal $\delta$: that is, $\po$ can be factored to the form $\po = \po_0*\po_1$, where $\po_0$ is nontrivial,  $|\po_0|<\delta$ and $\Vdash_{\po_0} "\po_1 \text{ is }(\delta+1)\text{-strategically closed."}$ Let $G\subseteq \po$ be generic over $V$, and assume that $j\colon V[G]\to M^*$ is an elementary embedding with critical point $\kappa > \delta$, such that $M^*\subseteq V[G]$ and $M^*$ is closed under $\delta$-sequences of its elements that belong to $V[G]$. Then:
		\begin{enumerate}
			\item $M^*$ has the form $M[H]$ where $M\subseteq V$ and $H = j(G)$ is $j(\po)$-generic over $M$. Furthermore, $M = V\cap M[H]$.
			\item If $j$ is definable in $V[G]$ from parameters, then $j\uhr V\colon V\to M$ is definable in $V$ from parameters.
		\end{enumerate}
		\end{theorem}
		
		We remark that $M^*$ having the form $M[H]$ already follows from Laver’s ground model definability theorem. However, the inclusion $M \subseteq V$ is nontrivial and need not hold for forcings without a gap (see, for example, \cite[Subsection 5.2]{GitikKaplanRestElm}).

		\begin{remark}
		Nonstationary support iterated forcings usually satisfy the assumptions of the Gap forcing theorem. Our splitting forcing $\po^\tau$, however, is a nonstationary-support product rather than an iteration. It is unclear whether the conclusions of the Gap Forcing Theorem apply directly to $\po^\tau$. To address this, we modify $\po^\tau$ by considering $\ro^\tau = \text{Cohen}(\omega)* \po^\tau$, which has a gap at $\delta = \omega_1$. An alternative approach would be to define $\po^\tau$ as a nonstationary-support iteration instead of a product, but this comes at the cost that the ultrapowers $\text{Ult}(V[G], U^*_\eta)$ for $\eta<\tau$ are not necessarily the same (this would still provide a proof for the original Friedman-Magidor theorem \ref{Theorem: Friedman–Magidor}).
		\end{remark}

		With this in mind, we now have all the ingredients to conclude the proof of Theorem \ref{Theorem: controlling normal measures}.
		
		\begin{proof}[Proof of Theorem \ref{Theorem: controlling normal measures}]
		Let $\ro^\eta = \text{Cohen}(\omega)* \po^\tau$. Let $G = g* G'\subseteq \ro^\tau$ be generic over $V$, where $g$ is Cohen-generic over $V$, and $G'$ is $\po^\tau$-generic over $V[g]$. By the Lévy–Solovay theorem \cite{LevySolovay1967measurable}, every normal measure on $\kappa$ in $V$ generates a normal measure on $\kappa$ in $V[g]$, and conversely, every normal measure on $\kappa$ in $V[g]$ arises in this way. Given $U\in V$ a normal measure, denote by $U'\in V[g]$ the normal measure it generates. As above, given $\eta<\tau$, let $S_\eta =\{ \alpha<\kappa \colon (\cup G')(\alpha) = f_\eta(\alpha) \}$. By Lemma \ref{Lemma: lifting normal measures}, for every $\eta<\tau$, $U'\cup \{ S_\eta \}$ generates a measure ${U'}^*_{\eta}\in V[g*G']$, and all the measures $\la {U'}^*_{\eta} \colon \eta<\tau \ra$ have the same ultrapower. We can slightly abuse the notation and denote ${U}^*_{\eta}$ instead of ${U'}^*_{\eta}$. Each such $U^*_\eta$ is generated from $U\cup \{ S_\eta \}$ in $V[g*G']$, since every set in $U'$ contains a set from $U$. This concludes the proof of clause 1 of Theorem \ref{Theorem: controlling normal measures}. 
		
		We proceed and prove clause 2. Let $W\in V[G]$ be a normal measure on $\kappa$, where, as before, $G$ is viewed as a generic $g*G'\subseteq \text{Cohen}(\omega)*\po^\tau$. By Hamkins' Gap forcing theorem, $U = W\cap V\in V$, since $U = \{ X\subseteq \kappa \colon \kappa\in (j_W\uhr V)(X) \}$. $U$ is normal because $W$ is normal in $V[G]$.  Denote $\eta = j_W(\cup G')(\kappa)$.  Using again the normality of $W$,  $\eta< j_W(f_\tau)(\kappa) = \tau$ and $S_\eta\in W$. So $U\cup \{ S_\eta \}\subseteq W$, and, by the previous paragraph,  $W = U^*_\eta$. 
		
		Finally, if GCH holds in $V$, then it continues to hold in $V[g]$, and over this model $\po^\tau$ preserves cardinals by Corollary \ref{Corollary: if the iterands preserve cards and GCH then the entire iteration peserves cards}.
		\end{proof}

We finish this section by proving that the splitting forcing $\po^\tau$ preserves the Mitchell order. As a corollary, $\ro^\tau$ preserves the Mitchell order as well.

\begin{lemma}
	Assume that $U_0\vartriangleleft U_1$ are normal measures on $\kappa$ in $V$. Fix $\eta,\zeta< \tau$ and $G\subseteq \po^\tau$ generic over $V$. Then $\left(U_0\right)^*_\eta\vartriangleleft \left( U_1 \right)^*_\zeta$. 
\end{lemma}

\begin{proof}
	Since $U_0 \vartriangleleft U_1$, we can lift $U_0$ to $\left(\left(U_0\right)^*_\eta\right)^{M_{U_1}[G]}$ by forcing with $\po^\tau$ over $M_{U_1}$. It suffices to prove that $\left(\left(U_0\right)^*_\eta\right)^{M_{U_1}[G]} = \left(U_0\right)^*_\zeta$.
	
	Both ultrafilters $\left(\left(U_0\right)^*_\eta\right)^{M_{U_1}[G]}$, $ \left(U_0\right)^*_\eta$ are generated by $U_0\cup \{ S_\eta \}$, but the former ultrafilter belongs to $V[G]$ and the latter belongs to $M_{U_1}[G]\subseteq V[G]$. So we just need to verify that $V[G]$ and $M_{U_1}[G]$ share the same subsets of $\kappa$. This is another classical application of the Fusion Lemma \ref{Lemma: Fusion for a product}.
	
	Let $\name{X}\in V$ be a $\po^\tau$-name for a subset of $\kappa$. For every $\alpha<\kappa$, let 
	$$e(\alpha) = \{ r\in \po^\tau \colon r\uhr \alpha+1 \Vdash \exists X_\alpha\subseteq \alpha \  \left( r\setminus (\alpha+1)\Vdash \name{X}\cap \alpha = X_\alpha \right)  \}.$$
	Each $e(\alpha)$ is a dense open subset of $\po^\tau$, since $\po^\tau\setminus \alpha+1$ is more than $\alpha$-closed.  By the Fusion Lemma \ref{Lemma: Fusion for a product}, there exists $p\in G$ and a club $C\subseteq \kappa$ such that for every $\alpha\in C$,
	$$ p\uhr \alpha+1 \Vdash \exists X_\alpha\subseteq \alpha \  \left( p\setminus (\alpha+1) \Vdash \name{X}\cap \alpha = X_\alpha \right). $$
	For each $\alpha\in C$, let $\name{X}_\alpha$ be a canonical $\po_{\alpha+1}$-name for $X_\alpha$ as a subset of $\alpha$. Note that $\la X_\alpha \colon \alpha\in C \ra\in M_{U_1}$ and also $p\in M_{U_1}$. Since $p\in G$, we can compute $(\name{X})_G$ in $M_{U_1}[G]$ from $G$ and $\la \name{X}_\alpha \colon \alpha\in C\ra$.
\end{proof}

\section{Lifting iterates of a normal measure}\label{Section: Lifting iterates of U}

Our next goal is to prove Theorem \ref{Theorem: Weak UA}, which provides a characterization of all possible lifts of $U^n$ for a normal measure $U$ on $\kappa$ and $n < \omega$.

We maintain the same notation as in the previous section, including the measurable cardinal $\kappa$ and the forcing notion $\po^{\tau}$, which we denote simply by $\po$ throughout this section.\footnote{In this section, there is no need to force with $\ro^\tau = \text{Cohen}(\omega) * \po^\tau$, as the main result is obtained by forcing over $L[U]$, and standard arguments eliminate the need for Hamkins’ Gap Forcing Theorem.}

In this section, we restrict attention to the case $\tau < \kappa$ only, due to technical difficulties that arise when applying the same proof technique in the cases $\tau \in \{\kappa, \kappa^+\}$.

We denote by $\kappa_i$ for $i < \omega$ the images of $\kappa$ under the iterated ultrapowers with $U$; that is, for each $i < \omega$, $\kappa_i = j_{U^i}(\kappa)$.

\begin{lemma}\label{Lemma: Multivariable fusion}
	Fix a normal measure $U$ on $\kappa$, $n<\omega$, and a sequence $\la \eta_0, \ldots, \eta_{n-1} \ra\in \tau^n$. Assume that for every $\vec{\nu} = \la \nu_0, \ldots, \nu_{n-1} \ra \in [\kappa]^{n}$, $e(\vec{\nu}) = e(\nu_0,\ldots, \nu_{n-1})\subseteq \po$ is a dense open set. Then for every $p\in \po$ there exists $r\geq p$ such that
	$$\{  \vec{\nu}\in [\kappa]^n \colon r\cup \{ \la \nu_0, \eta_0 \ra, \la \nu_1, \eta_1 \ra, \ldots,  \la \nu_{n-1}, \eta_{n-1} \ra \} \in e(\vec{\nu}) \}\in U^n .\footnote{Note that there exists a club $C$ disjoint from $\supp(r)$, and $[C]^n\in U^n$. So we can assume that $r\cup \{ \la \nu_0, \eta_0 \ra, \ldots,  \la \nu_{n-1}, \eta_{n-1} \ra \}$ is a function for a set of $\vec{\nu}$-s in $U^n$.} $$
\end{lemma}

\begin{proof}
	We begin with the case $n=1$. Let $p$ be a condition. By the Fusion Lemma \ref{Lemma: Fusion for a product}, there exists a condition $q \geq p$ and a club $C\subseteq \kappa$ such that for every $\nu\in C$, 
	$$ \{ r\in \po_{\nu+1} \colon r \cup  (q\setminus \nu+1)\in e(\nu) \} $$ 
	is a dense subset of $\po_{\nu+1}$ above $q\uhr \nu+1$. We can assume that $C\cap \supp(q)=\emptyset$ by shrinking $C$. Then for every $\nu\in C$, there exists a condition $r(\nu)\in \po_{\nu}$ extending $q\uhr \nu$ such that
	$$ r(\nu) \cup \{\la \nu, \eta_0 \ra\} \cup  (q\setminus \nu+1)\in e(\nu). $$
	Since $j_U(\po)\uhr \kappa = \po$ and $j_U(q)\uhr \kappa = q$, we deduce that $r = [\nu\mapsto r(\nu)]_U$ is a condition in $\po$ extending $q$. Furthermore, there exists a set $A\in U$ such that for every $\nu\in A$, $r\uhr \nu = r(\nu)$. Then for every $\nu\in A\cap C\in U$, 
	$$r\uhr \nu \cup \{\la \nu, \eta_0 \ra\} \cup  (r\setminus \nu+1)\in e(\nu).$$
	
	Now proceed by induction. Fix a sequence $\la \eta_0, \ldots, \eta_{n-1}, \eta_{n} \ra\in \tau^{n+1}$. Assume that for every $\la \nu_0, \ldots, \nu_{n-1} \ra\in [\kappa]^n$, and for every $\nu\in \kappa\setminus \nu_{n-1}$,
	$$ e(\nu_0, \ldots, \nu_{n-1}, \nu) \subseteq \po$$
	is a dense open set.
	
	Fix $\vec{\nu} = \la \nu_0, \ldots, \nu_{n-1}\ra\in [\kappa]^n$. Define set $e(\nu_0, \ldots, \nu_{n-1})$ consisting of conditions $r$ for which there exists a set $A\in U$, such that for every $\nu\in A\setminus \nu_{n-1}+1$, 
	$$ r\cup \{ \la \nu, \eta_n \ra \}\in e(\nu_0, \ldots, \nu_{n-1}, \nu)  . $$
	Applying the case $n=1$, with $\eta_n$ in the role of $\eta_0$, shows that each set $e(\nu_0, \ldots, \nu_{n-1})$ is a dense open subset of $\po$.

	Fix $p\in \po$. Apply the induction hypothesis on the sequence 
	$$\la e(\nu_0,\ldots, \nu_{n-1}) \colon \la \nu_0,\ldots, \nu_{n-1} \ra\in [\kappa]^n \ra.$$
	There are $r\geq p$ and $X\in U^n$ such that for every increasing sequence $\vec{\nu} = \la \nu_0, \ldots, \nu_{n-1} \ra\in X $,
	$$ r\cup \{  \la  \nu_0,\eta_0 \ra ,\ldots, \la  \nu_{n-1},\eta_{n-1} \ra   \ra  \}\in e(\vec{\nu}). $$
	In particular, for every such sequence $\la \nu_0,\ldots, \nu_{n-1} \ra$  there exists a set $A( \nu_0, \ldots, \nu_{n-1} )\in U$ such that for every $\nu_n\in A( \nu_0, \ldots, \nu_{n-1} )\setminus \nu_{n-1}+1$,
	$$ r\cup \{  \la  \nu_0,\eta_0)  \ra ,\ldots, \la  \nu_{n-1},\eta_{n-1})  \ra , \la  \nu,\eta_{n})  \ra  \}\in e(\nu_0, \ldots, \nu_{n-1},\nu_n). $$
	Let
	$$ Y = \{  \la \nu_0, \ldots, \nu_{n-1}, \nu \ra\in [\kappa]^{n+1} \colon \la \nu_0 ,\ldots, \nu_{n-1} \ra\in X \text{ and } \nu\in A(\nu_0,\ldots, \nu_{n-1}) \} $$
	It's not hard to see that $Y\in U^{n+1}$. Shrink $Y$ by intersecting it with $[C]^{n+1}$, where $C$ is some club disjoint from $\supp(r)$.  It follows that the condition $r$ has the desired property, witnessed by the shrinked version of $Y$ which belongs to $U^{n+1}$.
\end{proof}

\begin{corollary}\label{Corollary: generating lifts of U^n}
	Let $U$ be a normal measure on $\kappa$. Fix $n<\omega$ and $\tau< \kappa$. Let $G\subseteq \po^\tau$ be generic over $V$. Then for every sequence $\vec{\eta} = \la \eta_0, \ldots, \eta_{n-1} \ra \in  \tau^n$, the set $U^n \cup \{ S_{\vec{\eta}} \}$ generates a $\kappa$-complete ultrafilter in $V[G]$, where 
	$$ S_{\vec{\eta}} =  \{ \la \nu_0, \ldots, \nu_{n-1} \ra\in [\kappa]^n \colon \forall i< n, (\cup G)(\nu_i) =  \eta_i  \}.$$ 
\end{corollary}

\begin{proof}
		We construct a $j_{U^n}(\po)$-generic set over $M_{U^n}$. Let
		$$  H_{\vec{\eta}} = \{  q\in j_{U^n}(\po) \colon \exists p\in G \left(q\leq  j_{U^n}(p)\cup \{ \la \kappa, \eta_0 \ra, \ldots, \la \kappa_{n-1}, \eta_{n-1} \ra   \} \right) \}.$$
		We argue that $H_{\vec{\eta}}\subseteq j_{U^n}(\po)$ is generic over $M_{U^n}$.
		 
		First, note that for every $m\leq n$ and $p\in \po$, $\kappa_m \notin j_{U^n}(p)$ (since $\kappa_m$ belongs to $j_{U^n}(C)$ where $C\subseteq \kappa$ is any club disjoint from the support of $p$). Thus, $H_{\vec{\eta}}$ consists of legitimate conditions in $j_{U^n}(\po)$. Clearly any pair of conditions in $H_{\vec{\eta}}$ are compatible. Thus, we only need to argue that $H_{\vec{\eta}}$ is $j_{U^n}(\po)$-generic over $V$. This is an immediate corollary of Lemma \ref{Lemma: Multivariable fusion}: given any dense open set $E\subseteq j_{U^n}(\po)$ in $M_{U^n}$, let $\la e(\nu_0,\ldots, \nu_{n-1}) \colon \la \nu_0, \ldots, \nu_{n-1} \ra\in [\kappa]^n \ra$ be a sequence of dense open subsets of $\po$ such that
		$$ E = j_{U^n}\left(  \la \nu_0,\ldots, \nu_{n-1}\ra\mapsto e(\nu_0,\ldots, \nu_{n-1}) \right)\left( \kappa_0,\ldots, \kappa_{n-1} \right). $$
		Pick a condition $r\in G$ such that
		$$\{  \vec{\nu}\in [\kappa]^n \colon r\cup \{ \la \nu_0, \eta_0 \ra, \ldots,  \la \nu_{n-1},\eta_{n-1} \ra \} \in e(\nu_0, \ldots, \nu_{n-1}) \}\in U^n. $$
		Then $j_{U^n}(r)\cup \{ \la \kappa, \eta_0 \ra, \ldots, \la \kappa_{n-1}, \eta_{n-1} \ra  \} \in E\cap H_{\vec{\eta}}$, as desired. 
	\end{proof}
	
	As a corollary, we obtain a complete characterization of all possible lifts of $U^n$ in $V[G]$.
	\begin{corollary}\label{Corollary: characterization of W that extends U^n}
	Let $U$ be a normal measure on $\kappa$. Fix $n<\omega$ and $\tau< \kappa$. Let $G\subseteq \po = \po^\tau$ be generic over $V$. For every $\sigma$-complete ultrafilter $W\in V[G]$ such that $U^n \subseteq W$, there exists $\vec{\eta} = \la \eta_0, \ldots, \eta_{n-1} \ra\in \tau^n$ such that $W$ is the measure generated by $U^n \cup \{ S_{\vec{\eta}} \}$.
	\end{corollary}

	\begin{proof}
		Let $j_W \colon V[G]\to M[H] = M[j_W(G)]$ be the ultrapower embedding associated with $W$.
		
		Since $U^n \subseteq W$, $W$ concentrates on the set $[I]^n$ consisting of increasing sequences of inaccessibles (in $V$) of length $n$. Thus, $[Id]_W = \{ \lambda_0, \ldots, \lambda_{n-1} \}$ for some inaccessibles (of $M$) $\lambda_0<\ldots < \lambda_{n-1}$.
		
		It follows that the generic $H = j_W(G)$ assigns to each $\lambda_i$ (for $0\leq i \leq n-1$) some $\eta_i = (\cup H)(\lambda_i) < j_W(\tau) =\tau$. Thus, letting $\vec{\eta} = \la \eta_0, \ldots, \eta_{n-1} \ra\in \tau^{n}$, we deduce that $S_{\vec{\eta}}\in W$, and $W$ is the measure generated by $U^n \cup \{ S_{\vec{\eta}} \}$.
	\end{proof}

	All the results provided so far have not required the assumption that $V = L[U]$. However, to obtain a full characterization of all $\sigma$-complete ultrafilters, the assumption $V = L[U]$ is useful, since it implies that every $\sigma$-complete ultrafilter is Rudin–Keisler equivalent to one extending $U^n$ for some $n < \omega$. The proof of Theorem \ref{Theorem: Weak UA}, given below, explains this in detail.

	\begin{proof}[Proof of Theorem \ref{Theorem: Weak UA}]
		Assume $V = L[U]$ is the core model. Let $G\subseteq \po^\tau$ be generic over $L[U]$.
		
		We first show that every $\sigma$-complete ultrafilter $W \in L[U][G]$ on an ordinal $\chi$ satisfies $W \cap L[U] \in L[U]$. Indeed, the restriction $j_W \uhr L[U]$ is an iterated ultrapower $j_{U^m}$ of $U$ over $L[U]$ for some $m < \omega$. Consequently,
		$$ W\cap L[U] = \{ X\subseteq \chi \colon [Id]_W\in j_{U^m}(X) \}\in L[U]. $$
		
		For clause 1 of Theorem \ref{Theorem: Weak UA}, the argument proceeds exactly as in the proof of Theorem \ref{Theorem: controlling normal measures}: given a normal measure $W\in L[U][G]$ on $\kappa$, first observe that $W \cap L[U]$ inherits normality from $W$, so $U = W \cap L[U]$; then set $\eta = j_W(\cup G)(\kappa)$, note that $S_\eta \in W$, and conclude that $W$ is the normal measure $U^*_{\eta}$ generated by $U \cup \{S_\eta\}$.

		Let us proceed and prove clauses 2 and 3. Let $W\in V[G]$ be an arbitrary $\sigma$-complete ultrafilter in $V[G]$, concentrating on an ordinal $\chi$. By the previous paragraphs, $W\cap L[U]$ is a $\sigma$-complete ultrafilter in $L[U]$, and as such, it is Rudin-Keisler equivalent to $U^n$ for some $n<\omega$. Fix an injection $h \colon \chi\to [\kappa]^n$ witnessing this (namely, for every set $X\subseteq \chi$, $X\in W\cap L[U]$ if and only if $h[X]\in U^n$). By replacing $W$ with $h_{*}(W) = \{ X\subseteq [\kappa]^n \colon h^{-1}[X]\in W \}$, we can assume from now on $W$ concentrates on $[\kappa]^n$ and  $W\supseteq U^n$. This is harmless since our goal is characterizing $j_W$ and $\text{Ult}(L[U][G], W)$, which are invariant under Rudin-Keisler equivalence. By Corollary \ref{Corollary: characterization of W that extends U^n}, there exists $\vec{\eta} = \la \eta_0 ,\ldots, \eta_{n-1} \ra\in \tau^n$ such that $W$ is the measure generated by $U^n \cup \{S_{\vec{\eta}}\}$. 
		
		We proceed and prove that $\text{Ult}(L[U][G], W)$ is a finite iterated ultrapower of $L[U][G]$ via normal measures.

		Let $W^0 \in V[G]$ be the normal measure on $\kappa$ generated from $U \cup \{ S_{\eta_0} \}$ (in the notations of Lemma \ref{Lemma: lifting normal measures}). Denote $N^0 = \text{Ult}(V[G], W^0)$, and note that $N^0 = M_U[H_{\eta_0}]$, where $H_{\eta_0} = j_{W^0}(G)$ is generated from $(\cup j_U[G]) \cup \{ \langle \kappa, \eta_0 \rangle \}$.
		
		Over $N^0 = M_U[H_{\eta_0}]$, let $W^1$ be the normal measure on $\kappa_1$ generated from $j_U(U) \cup \{S^1_{\eta_1}\}$, where $S^1_{\eta_1} = \{ \alpha < \kappa_1 : (\cup H_{\eta_0})(\alpha) = \eta_1 \}$. Set $N^1 = \text{Ult}(N^0, W^1)$. We argue that $N^1$ has the form $M_{U^2}[H_{\la \eta_0, \eta_1 \ra}]$ where $H_{\la \eta_0, \eta_1 \ra} = j_{W_1}(H_{\eta_1})$ is the $j_{U^2}(\po)$-generic generated by $\{ (\cup j_{U^2}[G])\cup \{\la \kappa, \eta_0 \ra, \la \kappa, \eta_1 \ra  \}  \}$.
		
		 Indeed, inside $V[G]$, the ground of $\text{Ult}(V[G], W)$ for every normal measure $W \in V[G]$ is definable as $\text{Ult}(V, W \cap V)$ (recall that $V$ itself is definable in $V[G]$ by Laver's ground model definability theorem). Furthermore, $j_W(G)$ is the generic set generated by $(\cup j_{W\cap V}[G])\cup \{ \la \kappa, \eta \ra  \}$ where $\eta$ is the unique ordinal below $\tau$ such that $S_{\eta}\in W$. Having all of this in mind, we can now use elementarity to deduce how $N^1 = \text{Ult}(N^0, W^1)$ looks inside $N^0 = M_U[H_{\eta_0}]$: the ground of $\text{Ult}(N^0, W^1)$ is 
		\[
		\text{Ult}(M_U, W^1 \cap M_U) = \text{Ult}(M_U, j_U(U)) = M_{U^2}
		\]
		and the corresponding generic set $j_{W^1}(H_{\eta_0})$ is the one generated by 
		$$j_{W^1\cap M_U}[H_{\eta_0}]\cup (\{\la \kappa_1, \eta_1 \ra\} ).$$
		Since $W^1\cap M_U = j_{U}(U)$ and $H_{\eta_0}$ itself is generated by $j_U[G]\cup \{ \la \kappa, \eta_0 \ra \}$, we deduce that $j_{W^1}(H_{\eta_0})$ is the generic generated by
		$$  j_{U^2}[G]\cup (\{ \la \kappa, \eta_0 \ra , \la \kappa_1, \eta_1 \ra\}  $$
		which is exactly the generic $H_{\la \eta_0, \eta_1 \ra} $ defined above.
		
		Continuing this inductively, we define models $N^i$, normal measures $W^i\in N^{i}$, ordinals $\eta_i< \tau$ and generics $H_{\la \eta_0, \ldots, \eta_i \ra}\subseteq j_{U^i}(\po)$, such that for every $i\leq n-1$,
		\begin{enumerate}
			\item $N^{i} = \text{Ult}(N^{i}, W^i) \simeq M_{U^{i+1}}[H_{ \la \eta_0, \ldots, \eta_i \ra }]$.
			\item $W^{i+1}$ is generated inside $N^{i}$ from $j_{U^{i+1}}(U) \cup \{  S^{i+1}_{\eta_{i+1}} \}$ where $S^{i+1}_{\eta_{i+1}} = \{ \alpha<\kappa_{i+1} \colon (\cup H_{\la \eta_0, \ldots, \eta_i \ra})(\alpha) = \eta_{i+1} \}$.
			\item $H_{\la \eta_0, \ldots, \eta_{i+1} \ra}\subseteq j_{U^{i+1}}(\po)$ is generated from  $(\cup H_{\la \eta_0, \ldots, \eta_i \ra})\cup \{ \la \kappa_{i+1}, \eta_{i+1} \ra \}$.  Equivalently, it's generated from $j_{U^{i+1}}[G] \cup \{  \la \kappa, \eta_0 \ra , \ldots, \la \kappa_{i+1}, \eta_{i+1} \ra  \}$.
		\end{enumerate}
		
		Note that $H_{\langle \eta_0, \ldots, \eta_{n-1} \rangle}$ is the same generic from the proof of Corollary \ref{Corollary: generating lifts of U^n}, corresponding to the ultrapower embedding associated with the measure generated by 
		$U^n \cup \{ S_{\langle \eta_0, \ldots, \eta_{n-1} \rangle} \}$, which we have already identified as $W$. It follows that $N^{n-1} = M_{U^{n}}[H_{\langle \eta_0, \ldots, \eta_{n-1} \rangle}]$ 	is exactly $\text{Ult}(V[G], W)$. In other words, the ultrapower with $W$ is the result of the successive ultrapowers taken with the normal measures $W^0, W^1, \ldots, W^{n-1}$ described above. This proves Clause 2 of Theorem \ref{Theorem: Weak UA}.

		Finally, note that for every $i\leq n-1$, $N^{i} \simeq \text{Ult}(V[G], \left(U^*_0\right)^{i+1})$ (where $U^*_0$ is as in Lemma \ref{Lemma: lifting normal measures} for $\eta=0$). This follows by induction, using the fact that any pair of normal measures in $V[G]$ have the same ultrapower. Indeed, for $i = 0$, $N^0 \simeq \text{Ult}(V[G], W^0) \simeq \text{Ult}(V[G], U^*_0)$. Assuming the statement for $i$, note that $N^{i+1} \simeq \text{Ult}(N^i, W^i)$. But by the induction hypothesis, $N^i\simeq \text{Ult}(V[G], (U^*_0)^{i-1} )$, and $j_{(U^*_0)^{i-1}}( U^*_0 )$ is a normal measure on $\kappa_i$ in $N^i$. It follows that $N^{i+1} \simeq \text{Ult}(N^i, W^i)\simeq \text{Ult}(N^i, j_{(U^*_0)^{i-1}}( U^*_0 )) = \text{Ult}(V[G], (U^*_0)^{i} )$. 
		
		In particular, the above inductive argument shows that $N^{n-1}$ is isomorphic to $\text{Ult}(V[G], \left(U^*_0\right)^{n})$. This proves Clause 3 of Theorem \ref{Theorem: Weak UA}.
	\end{proof}

\section{Friedman–Magidor for extenders}\label{Section: Friedman–Magidor for extenders}

In the previous sections, our goal was to classify lifts of normal measures (and their iterations) in forcing extensions via the splitting forcing $\po^\tau$. We now turn to lifts of extenders. Throughout this section, we say that two extenders $E$ and $E'$ are equivalent if $j_E = j_{E'}$.

For the remainder of the section, fix a function $g \colon \kappa \to \kappa$ such that, for every $\alpha \leq \kappa$, $g(\alpha)$ lies below the least inaccessible cardinal above $\alpha$. Also fix the following more general version of the splitting forcing:
$$ \po = \{ f\in  \prod_{\alpha\in X} g(\alpha) \colon X = \dom(f) \text{ is nowhere stationary subset of }I \}. $$ 

As usual, we order $\po$ by inclusion, and define $\ro = \text{Cohen}(\omega)*\po$. 

Our forcing $\po$ can be viewed as a nonstationary support product $\prod_{\alpha\in I}^{NS} \qo_\alpha$, where $\qo_\alpha = \{ 0_{\qo_\alpha} \}\cup g(\alpha)$ is an atomic forcing, in which $g(\alpha)\subseteq \qo_\alpha$ is a maximal antichain. As in Sections \ref{Section: NS supports}, \ref{Section: The splitting forcing}, GCH ensures that $\po$ and $\ro$ preserve cardinals.

The function $g$ determines the number of (nonequivalent) lifts of a given ground model extender to the generic extension. The requirement that each $g(\alpha)$ be below the least inaccessible above $\alpha$ ensures that $\po$ satisfies the assumptions of the Fusion Lemma \ref{Lemma: Fusion for a product}. Our method provides control over all lifts of extenders $E$ such that $j_{E}(g)(\kappa)$ bounds all the generators of $E$ (we conjecture that those methods can be generalized to a wider class of extenders, see Remark \ref{Remark: Spaced NS supports}).

Our main goal in this section will be the following theorem:

\begin{theorem}\label{Theorem: Friedman–Magidor for extenders}
	Assume GCH, and let $g\colon \kappa\to \kappa, \po, \ro$ be as above. Let $G\subseteq \ro$ be generic over $V$. Then the following properties hold:
	\begin{enumerate}
		\item for every $(\kappa, \lambda)$-extender $E$ in $V$ whose generators are all below $j_E(g)(\kappa)$ and for each $\eta<j_E(g)(\kappa)$, there exists a $(\kappa,\lambda)$-extender $E^*_\eta\in V[G]$ such that $j_{E^*_\eta}\uhr V = j_E$. Furthermore, all the extenders $ \la E^*_\eta \colon \eta< j_E(\tau)(\kappa) \ra$ have the same ultrapower, and the extenders $\la E_\eta \colon \eta< j_E(g)(\kappa) \ra$ are nonequivalent.
		\item Every $(\kappa,\lambda)$-extender $E^* \in V[G]$ whose generators all lie below $j_{E^*}(g)(\kappa)$ and with $\lambda \geq j_{E^*}(g)(\kappa)$ is, up to equivalence, of the form $E^*_\eta$, for some $(\kappa,\lambda)$-extender $E \in V$ with generators below $j_E(g)(\kappa)$ and some $\eta < j_E(\tau)(\kappa)$.
	\end{enumerate}
\end{theorem}

We remark that the requirement $\lambda \geq j_{E^*}(g)(\kappa)$ in Clause 2 above is necessary. Indeed, Theorem \ref{Theorem: Classification of normal measures in the more general context} shows that even in the case $\lambda = \kappa+1$—that is, when $E^*$ is a normal measure on $\kappa$ in $V[G]$—the embedding $j_{E^*}\uhr  V$ may have a generator in the interval $(\kappa, j_{E^*}(g)(\kappa))$. Such a situation, however, cannot occur if $g(\alpha)\leq \alpha^+$ for all $\alpha$, as is also demonstrated in the proof of Theorem \ref{Theorem: controlling normal measures}.

\begin{remark}	\label{Remark: Spaced NS supports}
We conjecture that the requirement that $g(\alpha)$ lie below the least inaccessible above $\alpha$ can be relaxed by adopting a spaced nonstationary support.

Recall that the restriction on $g(\alpha)$ was imposed to ensure that the Fusion Lemma \ref{Lemma: Fusion for a product} applies to $\po$. Dropping this restriction, while still retaining a usable version of the Fusion Lemma, seems to require “spacing out” the nonstationary support—that is, requiring the support to be nonstationary only at inaccessibles from a specified stationary set.

For example, if $\kappa$ is sufficiently large, one can force with supports that are nonstationary at inaccessible cardinals which are limits of inaccessibles. In this setting, the quotients $\po \setminus (\alpha+1)$ (for $\alpha < \kappa$) exhibit stronger closure properties, making it possible to establish a version of the Fusion Lemma. With the bound on $g(\alpha)$ taken to be the least inaccessible limit of inaccessibles above $\alpha$, Theorem \ref{Theorem: Friedman–Magidor for extenders} should then apply to a broader class of extenders. We will not pursue this direction further in the present paper.
\end{remark}

\begin{lemma}\label{Lemma: Generics that lift an extender}
Assume that $E$ is a $(\kappa,\lambda)$-extender in $V$ whose generators all lie below $j_E(g)(\kappa)$. Let $G \subseteq \ro$ be generic over $V$. Then in $V[G]$ there are $\eta$ nonequivalent $(\kappa,\lambda)$-extenders $ \la E^*_\eta \colon \eta< j_E(g)(\kappa) \ra \in V[G]$, all with the same ultrapower model, such that each $j_{E^*_\eta}$ is a lift of $j_E$.
\end{lemma}

\begin{proof}
	Write $G = g' * G'$, where $g' \subseteq \mathrm{Cohen}(\omega)$ is generic over $V$, and $G' \subseteq \po$ is generic over $V[g']$. The extender $E$ naturally and uniquely lifts to an extender $E' \in V[g']$, namely the $(\kappa, \lambda)$-extender derived from $j'$, the standard lift of $j$ to an embedding from $V[g']$ into $M_E[g']$. From this point onward, we work over $V[g']$ and aim to show that $E'$ lifts in (at least) $\eta$ distinct ways.
	
	It suffices to prove that over any ground model $V$, whenever $G\subseteq \po$ is generic over $V$ and $E\in V$ is a $(\kappa, \lambda)$-extender, there are at least $j_E(g)(\kappa)$ nonequivalent $(\kappa, \lambda)$-extenders $\la E^*_\eta \colon \eta< j_E(g)(\kappa) \ra$, all of them with the same ultrapower, such that each $j_{E^*_\eta}$ lifts $j_E$.
	
	Thus, we change the settings to the ones mentioned in the previous paragraph. Fix $\eta< j_E(g)(\kappa)$. Define, in $V[G]$,
	$$ H_\eta = \{ q\in j_E(\po) \colon \exists p\in G \  \left(q\subseteq j_E(p)\cup \{ \la \kappa, \eta \ra \} \right) \}. $$
	We will argue that $H_\eta$ is $j_E(\po)$-generic over $V$. The crux of the argument lies in proving that $H_\eta$ meets every dense open set of $M_E$. 
	
	Fix a dense open subset $D\in M_E$ of $j_E(\po)$. Let $a\in [j_E(g)(\kappa)]^{<\omega}$ be a set of generators and $e\colon [\kappa]^{|a|}\to V$ be a function such that $D = j_E(e)(a)$. Define, for every $\alpha<\kappa$,
	$$ d(\alpha) = \{  q\in \po \colon q\uhr \alpha+1 \Vdash \forall b\in [g(\alpha)]^{|a|} \  \exists r\in \name{G}\uhr \alpha+1  \left( r^{\frown} (q\setminus (\alpha+1)) \right)\in e(b) \}. $$
	The fact that $g(\alpha)$ and $|\po_{\alpha+1}| = \max\{\alpha^{+},g(\alpha)\}$ are both strictly below the least inaccessible above $\alpha$, which is the amount of closure of $\po\setminus \alpha+1$, implies that $d(\alpha)\subseteq \po$ is dense open. Thus, we can apply the Fusion Lemma \ref{Lemma: Fusion for a product} and find a condition $p\in G$ and a club $C\subseteq \kappa$ such that for every $\alpha\in C$, 
	$$ \{ q\in \po_{\alpha+1} \colon q^{\frown} (p\setminus \alpha+1)\in d(\alpha) \} $$
	is dense open above $p\uhr \alpha+1$. In particular, for every such $\alpha\in C$,
	$$ p\uhr \alpha+1 \Vdash  \forall b\in [g(\alpha)]^{|a|} \  \exists r\in \name{G}\uhr \alpha+1  \left( r^{\frown} (p\setminus (\alpha+1)) \right)\in e(b). $$
	Since $C\subseteq \kappa$ is a club, $\kappa\in j_E(C)$, and thus
	$$ p^{\frown} 0_{\qo_\kappa} \Vdash \exists r\in j_U(\name{G})\uhr \kappa+1\  ( r^{\frown}(j_E(p)\setminus \kappa+1)\in D ) $$
	where we used the fact that $\bigcap_{b\in [j_E(g)(\kappa)]^{<\omega}} j_E\left( e \right)(b)  \subseteq E$. Since $G$ is $j_E(\po)\uhr \kappa = \po$-generic over $M_E$, and $\eta$ is $j_E(\la \qo_\alpha \colon \alpha<\kappa\ra)(\kappa)$-generic over $M_E[G]$, we can extend $p$ inside $G$ and achieve $j_E(p)\cup \{ \la \kappa, \eta \ra  \}\in E\cap H_\eta$.
	
	This proves that each $H_\eta$ is $j_E(\po)$-generic over $M_E$. Since $j_E[G]\subseteq H_\eta$, $j_E$ lifts to an embedding $j^*_{E,\eta} \colon V[G]\to M_E[H_\eta]$. By standard arguments, $j^*_{E,\eta}$ is the ultrapower embedding of the $(\kappa, \lambda)$-extender $E^*_\eta$ derived from it. The extenders $\la E^*_\eta \colon \eta<j_E(g)(\kappa) \ra$ are nonequivalent, since their elementary embeddings disagree on their value at $G$. Nevertheless, for every $\eta\neq \eta'$ below $j_E(g)(\kappa)$, $M_E[H_\eta] = M_E[H_{\eta'}]$, namely $E^*_\eta, E^*_{\eta'}$ have the same ultrapower model. Indeed, $H_{\eta'}$ can be easily computed inside $M_E[H_\eta]$ by changing the generic value of $H_\eta$ at coordinate $\kappa$ from $\eta$ to $\eta'$.
\end{proof}

\begin{remark}\label{Remark: killing a generator}
Note that $\eta$ is not a generator of $j^*_{E,\eta} = j_{E^*_\eta}$, even if $\eta$ was a generator of $E$, since $\eta = j_{E^*_\eta}( \cup G )(\kappa)$ in $V[G]$. In fact, this shows that $\po$ can be used to "kill" a generator of an extender $E\in V$. 

This implies that normal measures on $\kappa$ in the generic extension via $\po$ might rise from non-normal ground model measures (even though they restrict to a ground model normal measure, their ultrapower embedding lifts the ultrapower embedding of a non-normal measure). See Theorem \ref{Theorem: Classification of normal measures in the more general context} for more details on the characterization of normal measures in the new settings.
\end{remark}

\begin{theorem}\label{Theorem: Clause 2 of the theorem on extenders}
	Assume GCH. Let $G\subseteq \ro$ be generic over $V$. Let $E^*$ be a $(\kappa, \lambda)$-extender of $V[G]$ whose generators are all below $j_{E^*}(g)(\kappa)$ and assume that $\lambda \geq j_{E^*}(g)(\kappa)$. Then there exists a $(\kappa,\lambda)$-extender $E\in V$ and $\eta<j_E(g)(\kappa)$ such that $E^* $ is equivalent to $E^*_\eta$.
\end{theorem}

\begin{proof}
Denote $M[H]\simeq \text{Ult}(V[G], E^*)$. By Hamkins' Gap forcing theorem, $j_{E^*}\uhr V \colon V\to M$ is definable in $V$. Let $E$ be the $(\kappa, \lambda)$-extender derived from $j_{E^*}\uhr V$ in $V$.  

We argue first that $j_E = j_{E^*}\uhr V$ and $M = M_E$. Define $k\colon M_E\to M$ such that, for every $a\in [\lambda]^{<\omega}$ and $f\colon [\kappa]^{|a|}\to V$ in $V$, $k( j_E(f)(a) ) = j_{E^*}(f)(a)$. It suffices to prove that $k$ is the identity map.

\begin{claim}
	$k$ is elementary and $j_{E^*}\uhr V = k\circ j_E$.
\end{claim}

\begin{proof}
	Let $a\in [\lambda]^{<\omega}$ and assume that $f_0,f_1 \colon [\kappa]^{|a|}\to V$ are functions in $V$. Then:
	\begin{align*}
		j_E(f_0)(a)= j_E(f_1)(a) \iff & \{ \vec{\nu}\in [\kappa]^{|a|} \colon f_0(\vec{\nu}) = f_1(\vec{\nu}) \}\in E_a\\
		\iff & a\in j_{E^*}\left(\{ \vec{\nu}\in [\kappa]^{|a|} \colon f_0(\vec{\nu}) = f_1(\vec{\nu}) \}\right)\\
		\iff & j_{E^*}(f_0)(a) = j_{E^*}(f_1)(a).
	\end{align*}
Building on this argument, full elementarity of $k$ can be shown. 
\end{proof}

Note that $\text{crit}(k)\geq \lambda$ since every generator below $\lambda$ is mapped to itself. In fact, the critical point of $k$ is the least generator (if such exists) of $j_{E^*}\uhr V$ which is greater or equal $\lambda$. We argue that no such generator exists. Fix an ordinal $\zeta\geq \lambda$, and write, in $V[G]$, $\zeta = j_{E^*}(f)(a)$ for some $a\in [j_E(g)(\kappa)]^{<\omega}$ and $f\colon [\kappa]^{|a|}\to \text{Ord}$ in $V[G]$. We would like to prove here a strengthening of Corollary \ref{Corollary: evaluating new functions by old ones}:

\begin{claim}
	There exists a partial function $F\in V$ and a club $C\subseteq \kappa$ such that, for every $\alpha\in C$ and $\vec{\nu}\in [g(\alpha)]^{<\omega}$,
	$$ f(\vec{\nu})\in F(\alpha, \vec{\nu})  \text{ and } |F(\alpha, \vec{\nu})|\leq \max \{ \alpha^+, g(\alpha) \}.$$ 
\end{claim}

\begin{proof}
	For every $\alpha<\kappa$, define
	$$ d(\alpha) = \{ q\in \po \colon q\uhr \alpha+1 \Vdash \forall \vec{\nu}\in [g(\alpha)]^{<\omega} \ \exists \gamma_{\nu}\in \text{Ord} \ (q\setminus \alpha+1 \Vdash \name{f}(\vec{\nu}) = \check{\gamma}_{\nu} ) \}. $$
	Each $d(\alpha)$ is dense-open, since $\po\setminus \alpha+1$ is more than $g(\alpha)$-closed. By the Fusion Lemma \ref{Lemma: Fusion for a product} there exists $p\in G$ and a club $C\subseteq \kappa$ such that for every $\alpha\in C$ and $\vec{\nu}\in [g(\alpha)]^{<\omega}$, there exists a $\po_{\alpha+1}$-name $\name{\gamma}^{\alpha}_{\vec{\nu}}$ such that $ p\Vdash \name{f}(\vec{\nu}) = \name{\gamma}^{\alpha}_{\vec{\nu}} $. 
	Now, for every such $\alpha\in C$ and $\vec{\nu}\in [g(\alpha)]^{<\omega}$, let 
	$$F(\alpha, \vec{\nu}) = \{ \gamma\in \text{Ord} \colon \exists r\geq p\uhr \alpha+1 \  ( r\Vdash \name{\gamma}^\alpha_{\vec{\nu}} = \check{\gamma} )\}$$
	and note that $|F(\alpha, \vec{\nu})| \leq |\po_{\alpha+1}| = \max\{g(\alpha), \alpha^+\}$ due to GCH and the fact that $|\qo_\alpha| = g(\alpha)$.
\end{proof}

Now we are prepared to complete the proof that $k$ is the identity map. Let $C,F$ be as in the claim. Since $C$ is a club in $\kappa$, $\kappa\in j_{E^*}(C)$. Since $a\in [\lambda]^{<\omega}$, we have that 
$$\zeta = j_{E^*}(f)(a)\in (j_{E^*}\uhr V)(F)(\kappa, a) . $$
Since we initially picked $\zeta\geq \lambda \geq j_{E^*}(g)(\kappa)$, and $\kappa, a$ are all below $j_E(g)(\kappa)$, the above equality shows that $\zeta$ is not a generator of $j_{E^*}\uhr V$. As a corollary, we deduce that $k$ is the identity map, and $j_{E^*}\uhr V = j_{E}$. 

To sum up the situation so far, we have proved that $M = M_E$ and $j_{E^*}\uhr V = j_E$. Write $G = g'*G'$ where $g'\subseteq \text{Cohen}(\omega)$ is generic over $V$, and $G'\subseteq \name{\po}$ is generic over $V[g]$. Let $H = j_{E^*}(G)$ and write $H = g'*H'$, where $H'$ is $j_{E^*}({\name{\po}})$-generic over $M[g']$. Note that $j_E[G] \subseteq H$ since $j_{E^*}$ lifts $j_E$.

Denote $\eta = j_{E^*}(\cup G')(\kappa) < j_{E^*}(g)(\kappa)$. Again since $j_{E^*}$ lifts $j_E$, we deduce that $\eta< j_E(g)(\kappa)$, and thus the lift  $E^*_\eta\in V[G]$ is defined.

Let $j'\colon V[g']\to M[g']$ be the natural lift of $j$. Following the proof of Lemma \ref{Lemma: Generics that lift an extender}, the set 
$$ H'_\eta = \{  q\in j'(\po) \colon \exists p\in G' \left( q\subseteq j'(p)\cup \{ \la \kappa, \eta \ra \}   \right) \} $$
is $j'(\po)$-generic over over $M[g']$. Since $\eta = j_{E^*}(\cup G')(\kappa) = (\cup H')(\kappa)$ and $j'[G']\subseteq H'$, we obtain that $H'_\eta\subseteq H'$, namely $H'_\eta = H'$. Thus $\text{Ult}(V[G], E^*) = M[H] = M_E[g'*H'_\eta]$ and $j_{E^*} = j_{E^*_\eta}$. It follows that $E^*$ and $E^*_{\eta}$ are equivalent, as desired.
\end{proof}

We can now conclude the proofs of the main results of this section.

\begin{proof}[Proof of Thoerem \ref{Theorem: Friedman–Magidor for extenders}]
	The first clause of Theorem \ref{Theorem: Friedman–Magidor for extenders} follows from Lemma \ref{Lemma: Generics that lift an extender}, and the second clause follows from Theorem \ref{Theorem: Clause 2 of the theorem on extenders}.
\end{proof}

\begin{proof}[Proof of Theorem \ref{Theorem: Consistency strengths at the level of supercompacts with Arthur and Alejandro}]
	Fix $E\in V$ a $(\kappa, \kappa^{++})$ extender witnessing that $\kappa$ is $\kappa^{++}$-strong. For every $\alpha<\kappa$, let $g(\alpha) = \alpha^{++}$. Let $G\subseteq \ro$ be generic over $V$. By Theorem \ref{Theorem: Friedman–Magidor for extenders}, $E$ lifts in exactly $j_E(g)(\kappa) = (\kappa^{++})^{M_E} = \kappa^{++}$ ways to an extender of $V[G]$, and a standard argument (see, for instance,  \cite[Proposition 8.4]{cummings2009iteratedforcingsandelemembs}) shows that each of them witnesses that $\kappa$ is $(\kappa+2)$-strong in $V[G]$. 
	On the other hand, each $(\kappa, \kappa^{++})$-extender $E^*\in V[G]$ witnessing $(\kappa+2)$-strongness is, again by Theorem \ref{Theorem: Friedman–Magidor for extenders}, such a lift of some ground model $(\kappa, \kappa^{++})$-extender $E\in V$. By Hamkins' Gap-forcing theorem \ref{Theorem: Hamkins Gap forcing thm}, $M^{V[G]}_{E^*}\cap V = M_E$, and $V_{\kappa+2}\subseteq M^{V[G]}_{E^*}\cap V$. It follows that $V_{\kappa+2}\subseteq M_E$ and $E$ witnesses $(\kappa+2)$-strongness in $V$.
\end{proof}

Goldberg observed that even though there are possibly $\kappa^{+3}$-many $(\kappa, \kappa^{++})$-extenders (in the GCH context), only $\kappa^{++}$ of them may have the same ultrapower. In this sense, Theorem \ref{Theorem: the number of extenders on a strong} is optimal.

\begin{theorem}[Goldberg]\label{Theorem: Number of extenders with the same ultrapower}
	Denote by $\beta$ the number of ultrafilters on $\kappa$. Let $E$ be a $(\kappa, \lambda)$-extender. Then, up to equivalence, there are at most $\max\{\beta, \lambda\}$ extenders with the same ultrapower as $E$.
\end{theorem}

\begin{proof}
	Following \cite{Goldberg2024TheUniquenessofElemEmbs}, we say that an elementary embedding $i\colon V\to N$ is almost an ultrapower embedding if for every $X\subseteq N$ there exists $a\in N$ such that $X\subseteq \mathcal{H}^{N}(i[V]\cup\{a\})$ (here and below, $\mathcal{H}$ denotes the Skolem hull). 
	
	\begin{claim}
		Let $j\colon V\to M$ be an extender embedding, $i\colon V \to N$ almost an ultrapower embedding, and $k\colon N\to M$ an elementary embedding such that $j = k\circ i$. Then $i$ is an ultrapower embedding; that is, there exists $a\in N$ such that $N = \mathcal{H}^{N}(i[V]\cup\{a\})$.
	\end{claim}

	\begin{proof}
		Since $j$ is an extender embedding, there exists a set $A\in V$ such that $j(A)$ covers the set of (finite sequences of) generators of $j$. In particular, $M = \mathcal{H}^M\left( j[V]\cup j(A) \right)$. It follows that $N = \mathcal{H}^N\left( j[V]\cup i(A) \right)$: for every element $x\in N$, there exists a function $f\in V$ such that 
		$$M\models \exists a\in j(A) \  (k(x) = j(f)(a)),$$
		where we view $k(x)$, $j(A) = k(i(A))$ and $j(f) = k(i(f))$ as parameters. By elementarity of $k$, 
		$$N\models \exists a\in i(A)\ (x = i(f)(a)), $$
		So $x\in \mathcal{H}^N\left( j[V]\cup i(A) \right)$.

		Since $i$ is almost an ultrapower embedding, for every ordinal $\xi$ there exists $\alpha_\xi\in N$ such that $(V_\xi)^N \subseteq \mathcal{H}^N\left( j[V]\cup \{\alpha_\xi\} \right)$. Each such $\alpha_\xi$ belongs to $\mathcal{H}^N\left( j[V]\cup \{ a_\xi \} \right)$ for some $a_\xi\in i(A)$. Thus, there exists $a\in i(A)$ such that for unboundedly many ordinals $\xi$, $\alpha_\xi \in \mathcal{H}^N\left( j[V]\cup \{a\} \right)$. Thus, $N\subseteq \mathcal{H}^N\left( j[V]\cup \{a\} \right)$.
	\end{proof}

	Having this in mind, fix an extender $E$ as in the formulation of the Lemma.  By \cite[Theorem 3.7]{Goldberg2024TheUniquenessofElemEmbs}, for every extender $F$ with the same ultrapower as $E$ there exists an inner model $N$ (actually, $N = \mathcal{H}^M\left( j_E[V]\cup j_F[V] \right)$) and elementary embeddings $i_0\colon V\to N$, $i_1\colon V\to N$ and $k\colon N\to M_E$ such that $j_E = k\circ i_0 $, $j_F= k\circ i_1$, and $i_0, i_1$ are almost ultrapower embeddings. By the above claim, since $i_0, i_1$ factor into extender embeddings, both are ultrapower embeddings with respect to ultrafilters $Z_0, Z_1$ on $\kappa$, respectively. 
	
	We argue that $k$ is uniquely determined from $E$, a function $f\colon \kappa \to \kappa$ and some $a\in [\lambda]^{<\omega}$. Indeed, let $a\in [\lambda]^{<\omega}$ and $f\in V$ be such that $k\left( [Id]_{Z_0} \right) = j_E(f)(a)$. Then $Z_0 = f_{*} (E_a) = \{ X\subseteq \kappa \colon f^{-1}[X]\in E_a \}$, and $k$ maps each element of the form $i_0(g)([Id]_{Z_0})$ to $j_E(g)\left( j_E(f)(a) \right)$. Thus $k$ is uniquely determined from $E$, $f$ and $a$.
	
	Since $j_F = k\circ j_{Z_1}$, $F$ is uniquely determined from $Z_1$, $E$,  $f$ and $a\in [\lambda]^{<\omega}$. The set of all such parameters has size $\max\{\beta,\lambda\}$.
\end{proof}

We finish this section with a complete classification of normal measures in $V[G]$ for $G\subseteq \ro$ generic over $V$. By Remark \ref{Remark: killing a generator}, normal measures $W\in V[G]$ may rise-up from non-normal measures $E$ of $V$ which have an additional single generator above $\kappa$. In this case, the forcing $\po$ might "kill" this additional generator, and $E$ lifts to a normal measure which differs from those described in Theorem \ref{Theorem: controlling normal measures}.

We summarize this in the following theorem. 

\begin{theorem}\label{Theorem: Classification of normal measures in the more general context}
	Assume GCH. Let $G\subseteq \ro$ be generic over $V$. Suppose that $W\in V[G]$ is a normal measure on $\kappa$. Then there exists a $(\kappa, j_W(g)(\kappa))$-extender $E$ in $V$ such that $W = E^*_{\eta}$ for some $\eta< j_E(g)(\kappa)$. Furthermore, $E$ is equivalent to a $\kappa$-complete ultrafilter on $\kappa$, which is the one derived from $j_E$ using $(\kappa, \eta)$ as a seed.
\end{theorem}

\begin{proof}
	As usual, write $G = g'*G'$, where $g'\subseteq \text{Cohen}(\omega)$ is generic over $V$, and $G'\subseteq \po$ is generic over $V[g']$. 
	
	Denote $\lambda = j_W(g)(\kappa)< j_W(\kappa)$. Let $E^*$ be the $(\kappa, \lambda)$-extender derived from $j_W$. Note that $j_W = j_{E^*}$, since the only generator of $W$ is  $\kappa$, and it is below $j_W(g)(\kappa)$. By Theorem \ref{Theorem: Clause 2 of the theorem on extenders}, we deduce that $j_{W} = j_{E^*_\eta}$, for some $(\kappa, \lambda)$-extender $E$ in $V$ and $\eta< j_E(g)(\kappa)$. By the analysis in the proof of Theorem \ref{Theorem: Clause 2 of the theorem on extenders}, $\eta = j_{W}\left( \cup G' \right)(\kappa)$. 
	
	It remains to prove that every $x\in M_E = \text{Ult}(V,E) $ has the form $j_E(\kappa, \eta)$, and thus $E$ is equivalent to the measure derived from $j_E$ using $(\kappa, \eta)$ as a seed.\footnote{possibly, $E$ is already equivalent to a normal measure, and then $\kappa$ is the single generator of $j_E$ and $W$ is really a lift of a ground-model normal measure. This happens, for instance, if $j_E(g)(\kappa)\leq \kappa^+$.} 
	
	Fix $\zeta\neq \kappa, \eta$. Since $W$ is normal, there exists $f\in V[G]$ such that $\zeta = j_W(f)(\kappa)$.

	Work in $V[g']$. By the standard Fusion argument from the proof of Corollary \ref{Corollary: evaluating new functions by old ones}, there exists $p\in G'$ and a club $C\subseteq \kappa$ such that for every $\alpha\in C$,
	$$ p\uhr \alpha+1 \Vdash \exists \gamma\in \text{Ord} \  \left( p\setminus (\alpha+1) \Vdash \name{f}(\alpha) = \check{\gamma}\right).$$
	Let $j'\colon V[g']\to M_E[g']$ be the natural lift of $j_E$. Since $\kappa\in j'(C)$, we have in $M_E[g']$ that
	$$ p^{\frown} 0_{\qo_\kappa} \Vdash \exists \gamma\in \text{Ord} \ \left(  j'(p)\setminus (\kappa+1) \Vdash j'(\name{f})(\kappa) = \check{\gamma} \right). $$
	We can use the fact that $\eta< j_E(\tau)(\kappa)$ is a possible generic for $\qo_\kappa$, to deduce
	$$ p\Vdash \exists \gamma\in \text{Ord} \  \left( \{  (\kappa, \eta) \}\cup  j'(p)\setminus (\kappa+1)  \Vdash j'(\name{f})(\kappa) = \check{\gamma}\right). $$
	Overall, we proved that, in $M_E[g]$, there exists a $j'(\po)\uhr \kappa = \po$-name $\name{\gamma}$ for an ordinal, such that
	$$ \left( j'(p) \cup \{ (\kappa, \eta)  \} \right)\Vdash j'(\name{f})(\kappa) = \name{\gamma}. $$

	We can assume that $\name{\gamma}$ is the least $\po$-name as above, where, by "least", we mean here that we fixed in advance a well order $\mathcal{W}$ of $V_\kappa$, which induces a well order $\mathcal{W}'$ of $(V_\kappa)^{V[g']}$, and chose $\name{\gamma}$ least with respect to $j'(\mathcal{W}')$. It follows that $\name{\gamma}$ has the form $j'(h)(\kappa, \eta)$ for some function $h\in V[g']$.
	
	Let $A\in M_E$ be the set of all possible values of $\name{\gamma}$, namely
	$$A = \{ \beta\in \text{Ord}  \colon  \exists q\in \ro \ (q\Vdash \name{\gamma} = \check{\beta}) \}.$$
	Then $A$ has the form $j(h^*)(\kappa, \eta)$ for some function $h^*\in V$ (note that $\ro = j_E(\ro)\uhr \kappa$ in the above definition of $A$), and $|A| \leq \kappa^+$ since $|\ro| = \kappa^+$. 
	
	Recall that $p\in G'$ and $\eta = j_W(\cup G')(\kappa)$. Thus, in $\text{Ult}(V[G], W) = M_E[g'*j_W(G')]$, we have $\zeta = j_W(f)(\kappa)\in A$. 
	
	In particular, for some ordinal $\xi<\kappa^+$, 
	$\zeta$ is the $\xi$-th element of $A$ with respect to the least well-order of $A$ of order type $\leq \kappa^+$ (again, the well order is chosen minimal with respect to $j_E(\mathcal{W})$). Letting $f_\xi$ be the canonical function that represents $\xi$, we obtain that, in $M_E$, $\zeta$ is the $j_E(f_\xi)(\kappa)$-th element of $j_E(h^*)(\kappa, \eta)$, with respect to the $j_E(\mathcal{W})$-least well order of $j_E(h^*)(\kappa, \eta)$ of order type $\kappa^+$. 
	
	All of this shows that $\zeta = j_E(h')(\kappa, \eta)$ for some function $h'\in V$. Since $\zeta$ was an arbitrary ordinal outside of $\{ \kappa, \eta \}$, we deduce that $E$ is equivalent to the $\kappa$-complete ultrafilter derived from $j_E$ using $(\kappa, \eta)$ as seeds.
\end{proof}

\section{Open questions and concluding remarks}\label{Section: Open Qeustions}

In this final section, we discuss some open questions.

\begin{question}\label{Question: having kappa++ many normal measures and a single normal ultrapower}
Is it consistent that GCH holds, there are $\kappa^{++}$ many normal measures on $\kappa$, but only a single normal ultrapower?
\end{question}

The same proof technique of Theorem \ref{Theorem: controlling normal measures} does not appear to apply here (see Remark \ref{Remark: limitation on tau}), but we still conjecture that the answer is positive.

Note that a positive answer to Question \ref{Question: having kappa++ many normal measures and a single normal ultrapower} would exhibit behavior somewhat different from that of extenders. For example, assuming GCH, there are at most $\kappa^{++}$ many $(\kappa, \kappa^{++})$-extenders with the same ultrapower (by Theorem \ref{Theorem: Number of extenders with the same ultrapower}), even though the total number of $(\kappa,\kappa^{++})$-extenders can be $\kappa^{+3}$.

Next, recall that the Friedman–Magidor theorem generalizes well to the context of violation of GCH on a measurable. 

\begin{theorem}[Friedman and Magidor, \cite{FriMag09}]\label{Theorem: Friedman Magidor violating the GCH}
Let $V = L[\E]$ be the minimal extender model that contains a cardinal $\kappa$ which is $(\kappa+2)$-strong. Then for every $\tau<\kappa^{++}$, there exists a cardinal preserving forcing extension in which $\kappa$ is measurable, $2^{\kappa} = \kappa^{++}$ and there are exactly $\tau$ normal measures on $\kappa$. 
\end{theorem}

This motivates the question whether the methods developed in this paper can yield similar results.

\begin{question}\label{Question: Breaking GCH in Friedman–Magidor}
	 Assume that $\kappa$ is $(\kappa+2)$-strong in an arbitrary ZFC+GCH model $V$. Fix $\tau<\kappa^{++}$. Is there a  cardinal preserving forcing extension $V[G]$ such that:
	 \begin{enumerate}
	 	\item $\kappa$ is measurable and $2^\kappa = \kappa^{++}$ in $V[G]$.
	 	\item For every extender $E\in V$ witnessing that $\kappa$ is $(\kappa+2)$-strong, there are $\tau$ normal measures on $\kappa$ in $V[G]$, such that for each of them, say $U$, $j^{V[G]}_U$ lifts $j_E$.
	 	\item Every normal measure on $\kappa$ arises as a lift of some strong extender $E$ as in the previous clause.
	 \end{enumerate}
\end{question}

In \cite{BenNeriaKaplan2024KunenLike}, it was shown that the Ultrapower Axiom is consistent with the failure of GCH at a measurable cardinal, via the construction of an $L[U]$-like model: a model in which $\kappa$ is the unique measurable cardinal, carries a single normal measure $U$, and every other $\sigma$-complete ultrafilter is equivalent to $U^n$ for some $n < \omega$. This result motivates the following question:

\begin{question}\label{Question: Weak UA + Breaking GCH}
	Is the Weak Ultrapower Axiom and the negation of the Ultrapower Axiom consistent with the failure of GCH on a measurable?
\end{question}

A natural approach to Question \ref{Question: Weak UA + Breaking GCH} is to force with the splitting forcing over the $L[U]$-like model described above and analyze all possible lifts of $U^n$, as in Section \ref{Section: Lifting iterates of U}. Indeed, the proof of Theorem \ref{Theorem: Weak UA} confirms that this construction yields a model of the Weak UA. However, applying the splitting forcing to a ground model in which GCH already fails at $\kappa$ may collapse $\kappa^{++}$ and inadvertently restore GCH. In fact, this is precisely what happens when forcing with the splitting forcing over the $L[U]$-like model, as the following claim shows:

\begin{claim}\label{Claim: Splitting might collapse cardinals}
	Let $\po$ be the splitting forcing, and assume that $2^{\kappa} = \kappa^{++}$. Suppose there exists a collection $\mathcal{C}$ of clubs in $\kappa$, such that $|\mathcal{C}|=\kappa^{+}$ and every club $D\subseteq \kappa$ contains a club from $\mathcal{C}$.\footnote{For instance, this holds if $V$ is the generic extension of a GCH model $V_0$, such that every $f\colon \kappa\to \kappa$ in $V$ is dominated by some $g\colon \kappa\to \kappa$ in $V_0$.} Then the splitting forcing $\po$ collapses $\kappa^{++}$.\footnote{A similar argument shows that any other reasonable nonstationary support product or iteration collapses $\kappa^{++}$ under the same assumptions.} 
\end{claim}

\begin{proof}
	Let $\po$ be the splitting forcing, and let $G \subseteq \po$ be generic over $V$. We claim that in $V[G]$ there exists a surjection from $\mathcal{C}$ to $(\power(\kappa))^V$.
	
	We say that for $C \in \mathcal{C}$ and $A \subseteq \kappa$ in $V$, the generic filter $G$ codes $A$ via $C$ if, for every $i < \kappa$,
	$$
	i \in A \iff (\cup G)(C^+_i) = 1
	$$
	where $C^+_i$ denotes the $i$-th successor point of $C$. 
	
	It suffices to prove that for every $A\subseteq \kappa$ in $V$ there exists $C\in \mathcal{C}$ such that $G$ codes $A$ via $C$.

	Indeed, fix $A \subseteq \kappa$ in $V$ and a condition $p \in \po$. By assumption, there is a club $C \in \mathcal{C}$ disjoint from $\supp(p)$. The sequence $\la C^+_i : i < \kappa \ra$ is nowhere stationary, so we may extend $p$ to a condition $q \in \po$ as follows:
	$$
	q = p \cup \{ \la C^+_i, 1 \ra : i \in A \} \cup \{ \la C^+_i, 0 \ra : i \notin A \}.
	$$
	By density, there exists $q \in G$ forcing that $\name{G}$ codes $A$ via some $C \in \mathcal{C}$.
\end{proof}

Claim \ref{Claim: Splitting might collapse cardinals} naturally leads to the following question.

\begin{question}
	Is there a variant of the splitting forcing that preserves $\kappa^{++}$ even when $2^\kappa = \kappa^{++}$ in $V$?
\end{question}

Finally, the following central question, raised by Friedman and Magidor in \cite{FriMag09}, has remained open ever since:

\begin{question}\label{Fine, normal ultrafulters}
	Is there a Friedman–Magidor theorem for fine, normal measures on $\power_\kappa(\lambda)$ for $\lambda>\kappa$?
\end{question}

\bibliographystyle{plain}
\bibliography{Bibliography}

\end{document}